\documentclass[11pt,a4paper]{article}

\usepackage[utf8]{inputenc}
\usepackage[T1]{fontenc}
\usepackage{amsmath,amssymb,amsthm}
\usepackage{mathtools}
\usepackage{enumitem}
\usepackage{booktabs}
\usepackage{array}
\usepackage{hyperref}
\usepackage{aliascnt}
\usepackage[capitalise,noabbrev]{cleveref}

\usepackage[margin=2.5cm]{geometry}

\newtheorem{theorem}{Theorem}[section]
\newaliascnt{lemma}{theorem}\newtheorem{lemma}[lemma]{Lemma}\aliascntresetthe{lemma}
\newaliascnt{proposition}{theorem}\newtheorem{proposition}[proposition]{Proposition}\aliascntresetthe{proposition}
\newaliascnt{corollary}{theorem}\newtheorem{corollary}[corollary]{Corollary}\aliascntresetthe{corollary}
\theoremstyle{definition}
\newaliascnt{definition}{theorem}\newtheorem{definition}[definition]{Definition}\aliascntresetthe{definition}
\newaliascnt{example}{theorem}\newtheorem{example}[example]{Example}\aliascntresetthe{example}
\theoremstyle{remark}
\newaliascnt{remark}{theorem}\newtheorem{remark}[remark]{Remark}\aliascntresetthe{remark}
\crefname{theorem}{Theorem}{Theorems}\crefname{lemma}{Lemma}{Lemmas}
\crefname{proposition}{Proposition}{Propositions}\crefname{corollary}{Corollary}{Corollaries}
\crefname{definition}{Definition}{Definitions}\crefname{example}{Example}{Examples}
\crefname{remark}{Remark}{Remarks}

\newcommand{\Fq}{\mathbb{F}_q}
\newcommand{\Fp}{\mathbb{F}_p}
\newcommand{\Zn}[1]{\mathbb{Z}_{#1}}

\newcommand{\Z}{\mathbb{Z}}
\newcommand{\R}{\mathbb{R}}
\newcommand{\Dn}{D^{n}}
\newcommand{\abs}[1]{\lvert #1 \rvert}

\newcommand{\floor}[1]{\lfloor #1 \rfloor}
\newcommand{\ceil}[1]{\lceil #1 \rceil}
\newcommand{\Uu}{\mathcal{U}}

\DeclareMathOperator{\Fix}{Fix}
\renewcommand{\Im}{\operatorname{Im}}
\DeclareMathOperator{\ord}{ord}
\DeclareMathOperator{\dmin}{d_{\min}}
\DeclareMathOperator{\Ef}{Ef}

\title{An Algebraic Framework for Data Systems:\\
Classification and Optimization of Algebraic Structures\\
for Data Organization and Coding}

\author{Kendy Inoa\thanks{Based on the author's M.Sc.\ thesis,
Pontificia Universidad Cat\'olica Madre y Maestra (PUCMM), 2025.
An earlier version was presented as a poster at the
2026 International Congress of Mathematicians (ICM 2026),
July 22, 2026, Pennsylvania Convention Center,
Philadelphia, PA, USA
(DOI: \texttt{10.5281/zenodo.21437132}).}\\[4pt]
\small Pontificia Universidad Cat\'olica Madre y Maestra (PUCMM)\\
\small Independent Researcher. Santo Domingo, Dominican Republic\\
\small ORCID: 0009-0001-4182-6859}

\date{}

\begin{document}
\maketitle

\begin{abstract}
Modern data systems are commonly studied through computational and
information-theoretic methods, while their algebraic properties remain
largely unexplored.  This paper introduces a mathematical framework
for modelling data systems using algebraic codes over finite groups,
organized by an \emph{axiom signature}: each algebraic axiom of the
data domain is shown to be \emph{sufficient} for a specific
coding-theoretic capability, and each implication is given its
precise reading (MDS parameters do not require a field; quotients
require normality, not commutativity).
Three types of results are presented.
\emph{Proved:} a Gilbert--Varshamov existence bound over arbitrary
finite groups, commutative or not; explicit direct-sum codes over the
elementary divisors of any finite abelian group; that transactions
preserving a validity predicate form a submonoid of the full
transformation monoid, which is not a group; that schema-preserving
transformations form a finite group whose orbits are counted by
Burnside's lemma; and that the constraint-preserving bijections of a
data system form a finite group whose orbits partition the valid
configurations.
\emph{Demonstrated:} an explicit code construction over
$\mathbb{F}_5$, with known codes over $\mathbb{Z}_4$ and
$(2^U,\triangle)$ cited to situate the framework.
\emph{Proposed:} an encoding efficiency measure Ef and an
optimization functional $\Phi$ with tunable weights, both evaluated at
declared coding instances.  Weight-independent comparison holds under
componentwise dominance; in general the weight simplex decomposes into
polyhedral chambers on which the induced order is constant.
\end{abstract}



\section{Introduction}
\label{sec:intro}

Classical coding theory operates within the well-established framework
of linear codes over finite fields $\Fq$
\cite{macwilliams1977,van_lint_1999,huffman2003}.  While this framework has
achieved remarkable depth and elegance, it does not directly address
the algebraic properties inherent in modern data structures: database
transactions, relational schemas, set-valued records, and symbolic
data sets carry natural algebraic operations—addition modulo $n$,
symmetric difference, function composition—that remain largely
unexplored from a coding-theoretic standpoint.

This paper addresses the following foundational question:

\begin{quote}
\emph{Given a finite set $D$ equipped with a binary operation $\star$,
under what conditions can one construct error-correcting codes
over $(D,\star)$ that inherit the algebraic structure of the data
domain?}
\end{quote}

We give constructive sufficient conditions and existence results; the
general question remains open (\cref{sec:open}).  The main
contributions are:
\begin{enumerate}[label=(\roman*)]
  \item a formal definition of algebraic codes over arbitrary finite
        group structures, together with an algebraic form of the
        Hamming distance that coincides with it numerically and makes
        translation invariance explicit
        (\cref{sec:framework});
  \item existence and construction theorems for codes over finite
        abelian groups, stratified into an optimal regime (groups
        isomorphic to additive groups of finite fields, yielding MDS
        codes) and a general regime (arbitrary abelian groups, yielding
        direct-sum codes), together with a Gilbert--Varshamov bound
        over arbitrary finite groups (\cref{sec:existence});
  \item a systematic classification methodology comprising an encoding
        efficiency measure $\Ef$ and an optimization functional $\Phi$
        with tunable weights, together with an instance-dominance
        criterion for weight-independent comparison and a chamber
        decomposition of the weight simplex
        (\cref{sec:classification}); and
  \item three original results connecting abstract algebra with data
        system theory: the monoid structure of ACID transactions, a
        novel application of Burnside's lemma to database schema
        equivalence, and a Reed--Solomon labelling of the orbit classes
        of valid configurations under a schema transformation group
        (\cref{sec:results}).
\end{enumerate}

When $(D,\star) = (\Fq, +)$, the definitions reduce to additive codes
over $\Fq$, and to linear codes under the scalar closure of
\cref{rem:z-module}.  The novel content lies in the systematic treatment
of structures beyond finite fields—cyclic groups, power-set algebras
under symmetric difference, monoids of transactions—and in the formal
connections established between group-theoretic properties and
coding-theoretic performance.  Group codes go back to
Slepian \cite{slepian1968}, and signal sets matched to groups were
studied by Loeliger \cite{loeliger1991}.  Prior work on codes over non-abelian
groups includes metacyclic codes \cite{sabin1995} and recent
constructions \cite{pillado2016}; for codes over finite abelian
groups, see \cite{dougherty2024}.

\begin{remark}[On the level of generality]
\label{rem:modules}
The definitions of this paper (\cref{sec:framework}) are
stated for finite groups, without assuming commutativity.
The existence and construction results
(\crefrange{sec:existence}{sec:classification}) require
abelian hypotheses where needed and are stated accordingly.
Every finite abelian group carries a canonical $\Z$-module
structure, and every group homomorphism between abelian groups is a
$\Z$-module homomorphism; the underlying theory thus extends to
modules over commutative rings,
and the author's thesis \cite{inoa_thesis} develops the full
module-theoretic framework.  The specialization to finite groups
adopted here yields the most explicit constructions and
computable instances of the efficiency functional $\Phi$.
\end{remark}

\paragraph{Organization.}
\cref{sec:framework} establishes the algebraic framework and
core definitions.  \cref{sec:existence} develops existence
theorems and metric preservation.  \cref{sec:classification}
presents the classification methodology and optimization functional.
\cref{sec:results} contains the three main original results.
\cref{sec:open} states open problems.

\section{Algebraic Framework for Data Coding}
\label{sec:framework}

\subsection{Fundamental Definitions}

\begin{definition}[Algebraic data structure]
\label{def:ads}
Let $D$ be a finite nonempty set of data elements.
An \emph{algebraic data structure} is a pair $(D,\star)$
in which $(D,\star)$ is a finite group.
\end{definition}

\begin{definition}[Algebraic code over data]
\label{def:code}
Let $(D,\star)$ be an algebraic data structure with identity
element $e$, and let $(M,\circ)$ be a finite group
(the \emph{message space}).
An \emph{algebraic code over data} is a pair $(C,\varphi)$
where $C \subseteq \Dn$ and
\[
  \varphi \colon M \longrightarrow \Dn
\]
is an injective group homomorphism satisfying the structural
preservation property:
\begin{equation}\label{eq:hom}
  \varphi(m_1 \circ m_2)
    = \varphi(m_1) \star \varphi(m_2)
  \quad\text{for all } m_1, m_2 \in M,
\end{equation}
where operations in $\Dn$ are performed componentwise.
The code is $C = \Im(\varphi)$.
\end{definition}

\begin{remark}[Algebraic hierarchy]
\label{rem:z-module}
\cref{def:ads} places no restriction on commutativity:
$(D,\star)$ may be abelian or non-abelian.
The definitions of algebraic code (\cref{def:code})
and algebraic distance (\cref{def:distance}) require
only the group axioms (closure, associativity, identity, inverses);
commutativity is not used.

The constructions and existence results of
\crefrange{sec:existence}{sec:classification} require
progressively stronger hypotheses, reflecting the following
algebraic hierarchy:
\begin{center}
\small
\begin{tabular}{@{}p{2.2cm}p{4.2cm}p{5.8cm}@{}}
\toprule
Level & Structure & What it enables \\
\midrule
General
  & Finite group $(D,\star)$
  & Distance $d^{\star}_H$, isometry, code definition \\[3pt]
Intermediate
  & Finite abelian group ($= \Z$-module)
  & Syndrome decoding, quotient $\Dn\!/C$ as group,
    direct-sum codes \\[3pt]
Specialized
  & $D \cong \Fq$ (finite field)
  & Polynomial evaluation $\to$ RS $\to$ MDS \\[3pt]
\midrule
Full generality
  & $A$-module over commutative ring $A$
  & Linear codes over rings
    (cf. $\Zn{4}$-linearity \cite{hammons1994}) \\
\bottomrule
\end{tabular}
\end{center}
Every abelian group is naturally a $\Z$-module
(with $n \cdot d = d \star d \star \cdots \star d$),
so the encoding map $\varphi$ of \cref{def:code}
is equivalently a $\Z$-module homomorphism when $(D,\star)$
is abelian.  When $(D,\star)$ admits a field structure
— that is, when $D \cong \Fq$ for some prime power $q$
— an algebraic code becomes an \emph{additive} code over $\Fq$.
It is $\Fq$-linear only if it is in addition closed under scalar
multiplication by every element of $\Fq$: an additive map in
characteristic $p$ preserves $\Fp$-scalars, not
$\Fq$-scalars.  The inclusion
$(\mathbb{F}_2,+) \hookrightarrow (\mathbb{F}_4,+)$ is injective and
additive, yet its image $\{0,1\}$ is not closed under multiplication
by $\omega$.  Closure under scalars depends on the external action of
the ring or module and is not supplied by the group operation; under
that additional hypothesis one recovers the classical notion of a
linear $[n,k]$ code with generator map $\varphi(u) = uG$.
The author's thesis \cite{inoa_thesis} develops the full
$A$-module framework; the present paper works primarily at the
abelian group level, where the most explicit constructions and
computable instances of the efficiency functional $\Phi$ are
obtained.
\end{remark}

\begin{definition}[Algebraic Hamming distance]
\label{def:distance}
Let $(D,\star,e)$ be a finite group with identity $e$.
For $x, y \in \Dn$, the \emph{algebraic Hamming distance}
is
\begin{equation}\label{eq:dist}
  d^{\star}_H(x,y)
    = \bigl\lvert
        \bigl\{
          i \in \{1,\ldots,n\}
          : x_i \star y_i^{-1} \neq e
        \bigr\}
      \bigr\rvert.
\end{equation}
This is the number of coordinates at which the ``difference''
$x_i \star y_i^{-1}$ is nontrivial.
\end{definition}

\begin{remark}\label{rem:hamming}
Since $x_i \star y_i^{-1} = e$ if and only if $x_i = y_i$, one has
$d^{\star}_H(x,y) = \abs{\{i : x_i \neq y_i\}}$: the algebraic
Hamming distance coincides with the classical Hamming distance for
\emph{every} finite group, abelian or not.  It is therefore not a
numerical generalization of the latter; its value lies in the
algebraic form, which makes translation invariance manifest and
exposes which axiom each construction consumes.  In the particular
case $(D,\star) = (\Fq, +)$ one has
$x_i \star y_i^{-1} = x_i - y_i$ and the classical formulation is
recovered verbatim.  The presence
of the inverse $y_i^{-1}$ in \eqref{eq:dist} is essential:
without it, the expression $x_i \star y_i$ may be nontrivial
even when $x_i = y_i$, unless every element of $D$ is
self-inverse (i.e., $D$ has exponent $2$).
\end{remark}

\begin{proposition}[Inheritance of algebraic properties]
\label{prop:inheritance}
Let $(D,\star)$ be an algebraic data structure and
$C \subseteq \Dn$ an algebraic code.
If $(D,\star)$ satisfies an algebraic property $\mathcal{P}$
(associativity, commutativity, existence of an identity),
then $(C,\star)$ satisfies $\mathcal{P}$ under the
componentwise extension of $\star$ to $\Dn$.
\end{proposition}

\begin{proof}
Immediate from the componentwise definition.  For instance,
associativity: $(a \star b) \star c = a \star (b \star c)$
because $(a_i \star b_i) \star c_i = a_i \star (b_i \star c_i)$
holds in each coordinate by the associativity of $(D,\star)$.
Commutativity and the existence of an identity follow by the
same coordinate-by-coordinate argument.
\end{proof}

\subsection{Weighted Metric Structure}
\label{subsec:metric}

The algebraic Hamming distance of \cref{def:distance}
counts differing positions uniformly.  A finer metric arises
when one assigns weights to group elements.

\begin{lemma}[Weighted algebraic metric]
\label{lem:metric}
Let $(D,\star)$ be a finite group with identity $e$,
and let $\delta \colon D \to \R_{\geq 0}$ be a function
satisfying:
\begin{enumerate}[label=(\roman*)]
  \item $\delta(e) = 0$ and $\delta(d) > 0$ for all
        $d \neq e$;
  \item $\delta(d_1 \star d_2) \leq \delta(d_1) + \delta(d_2)$
        for all $d_1, d_2 \in D$;
  \item $\delta(d^{-1}) = \delta(d)$ for all $d \in D$.
\end{enumerate}
Then the function
\begin{equation}\label{eq:weighted-dist}
  d^{\star}(x,y)
    = \sum_{i=1}^{n} \delta\bigl(x_i \star y_i^{-1}\bigr)
\end{equation}
is a metric on $\Dn$.  Moreover, the group operation is an
isometry:
\begin{equation}\label{eq:isometry}
  d^{\star}(a \star c,\, b \star c) = d^{\star}(a,b)
  \quad\text{for all } a,b,c \in \Dn.
\end{equation}
\end{lemma}

\begin{proof}
\emph{Non-negativity and identity of indiscernibles.}
Each summand $\delta(x_i \star y_i^{-1}) \geq 0$ by
construction, and $d^{\star}(x,y) = 0$ if and only if
$x_i \star y_i^{-1} = e$ for all $i$, i.e., $x = y$.

\emph{Symmetry.}  By (iii),
\[
  d^{\star}(x,y)
    = \sum_i \delta(x_i \star y_i^{-1})
    = \sum_i \delta\bigl((x_i \star y_i^{-1})^{-1}\bigr)
    = \sum_i \delta(y_i \star x_i^{-1})
    = d^{\star}(y,x).
\]

\emph{Triangle inequality.}  Write
$x_i \star z_i^{-1}
  = (x_i \star y_i^{-1}) \star (y_i \star z_i^{-1})$.
Applying (ii) coordinate by coordinate:
\[
  d^{\star}(x,z)
    = \sum_i \delta(x_i \star z_i^{-1})
    \leq \sum_i
      \bigl[\delta(x_i \star y_i^{-1})
            + \delta(y_i \star z_i^{-1})\bigr]
    = d^{\star}(x,y) + d^{\star}(y,z).
\]

\emph{Isometry.}
\[
  d^{\star}(a \star c,\, b \star c)
    = \sum_i \delta\bigl(
        (a_i \star c_i) \star (b_i \star c_i)^{-1}
      \bigr)
    = \sum_i \delta\bigl(
        a_i \star c_i \star c_i^{-1} \star b_i^{-1}
      \bigr)
    = \sum_i \delta(a_i \star b_i^{-1})
    = d^{\star}(a,b).
\]
The cancellation $c_i \star c_i^{-1} = e$ uses the group
axioms of $(D,\star)$; commutativity is used nowhere, and the
isometry \eqref{eq:isometry} is with respect to right translation.
\end{proof}

\begin{remark}\label{rem:discrete-weight}
Taking $\delta$ to be the discrete weight
$\delta(d) = \mathbf{1}_{d \neq e}$ recovers the algebraic
Hamming distance of \cref{def:distance}:
$d^{\star} = d^{\star}_H$.  Thus
\cref{lem:metric} subsumes
\cref{def:distance} as a special case.
\end{remark}

\section{Existence and Construction}
\label{sec:existence}

The construction of algebraic codes over $(D,\star)$ depends on
the algebraic type of $D$.  We distinguish two cases, corresponding
to the two regimes identified in the introduction: the \emph{field
case}, where $(D,\star)$ is isomorphic to the additive group of a
finite field and MDS codes exist; and the \emph{general abelian
case}, where direct-sum codes are obtained via the Fundamental Theorem
of Finite Abelian Groups.

\subsection{A General Existence Bound}
\label{subsec:gv}

\begin{theorem}[Gilbert--Varshamov bound over finite groups]
\label{thm:gv}
Let $(D,\star)$ be a finite group with $\abs{D} \geq 2$, let
$n \geq 1$, and let $1 \leq \Delta \leq n$ be a target distance.
Then there exists $C \subseteq \Dn$ with
$d^{\star}_H(c_1,c_2) \geq \Delta$ for all distinct
$c_1,c_2 \in C$ and
\[
  \abs{C} \;\geq\; \frac{\abs{D}^{\,n}}{V(n,\Delta-1)},
  \qquad
  V(n,\varrho)
  = \sum_{j=0}^{\varrho} \binom{n}{j}\bigl(\abs{D}-1\bigr)^{j},
\]
where $V(n,\varrho)$ is the size of a ball of radius $\varrho$
in $(\Dn, d^{\star}_H)$.
\end{theorem}

\begin{proof}
Select codewords greedily.  A word at distance $j$ from a given $c$
is obtained by choosing $j$ of the $n$ coordinates and altering
each to one of the $\abs{D}-1$ remaining values, so
$\abs{\{x : d^{\star}_H(x,c) \leq \varrho\}} = V(n,\varrho)$,
independently of $c$.  Each chosen codeword therefore forbids at
most $V(n,\Delta-1)$ elements.  While the forbidden volume is
strictly less than $\abs{D}^n$, some element remains uncovered and
may be added; the process cannot stop before
$\abs{C}\,V(n,\Delta-1) \geq \abs{D}^n$.
\end{proof}

\begin{remark}
\label{rem:gv-scope}
No hypothesis beyond invertibility is used, so the bound holds
over $S_3$ and $D_4$ exactly as over $\mathbb{F}_{16}$.  The set
produced is a subset of $\Dn$ and need not be a subgroup, so it is
an algebraic code in the sense of \cref{def:code} only
when the selection can be arranged to close under $\star$; the
statement is therefore about packings, and a group-code refinement
of it would be strictly stronger.  The bound is existential and
supplies no construction, which is why it bears on \textbf{P1}
without settling it.
\end{remark}

\subsection{The Field Case: MDS Codes}

\begin{theorem}[Existence of algebraic codes — field case]
\label{thm:field-case}
Let $(D,\star)$ be a finite abelian group whose order
$\abs{D} = q = p^{s}$ is a prime power, and suppose that
$(D,\star) \cong (\Fq,+)$.  Let $\mathcal{E}$ be a set of $n$
pairwise distinct evaluation points in $\Fq \cup \{\infty\}$, so
that $n \leq q+1$.  Then for every $1 \leq k \leq n$ there exists a
code $C \subseteq \Dn$ with parameters
\[
  n \leq q+1, \quad k \leq n, \quad d = n - k + 1,
\]
attaining the Singleton bound.  The code can be efficiently
decoded in time $O(n^2)$.
\end{theorem}

\begin{proof}
Fix a group isomorphism $\psi \colon D \to \Fq$.  Since $\psi$
is a bijection preserving the group operation, it preserves the
Hamming metric:
$d^{\star}_H(x,y) = d_H(\psi(x), \psi(y))$
for all $x, y \in \Dn$, where $d_H$ denotes the classical
Hamming distance on $\Fq^n$.

Write $\mathcal{E} = \{a_1,\ldots,a_n\}$ for the chosen set of $n$
pairwise distinct evaluation points and define the Reed--Solomon
code (introduced in \cite{reed1960} for $q = 2^{s}$ and
$\mathcal{E} = \Fq$)
\[
  C' = \bigl\{
    \bigl(f(a_1), f(a_2), \ldots, f(a_n)\bigr)
    : f \in \Fq[x],\; \deg f < k
  \bigr\}.
\]
Here $f(\infty)$ denotes the coefficient of $x^{k-1}$ in $f$, a
linear functional of $f$; the leading coefficient of variable degree
would not do, since it is not additive (over $\mathbb{F}_5$ the leading
coefficient of $x+1$ is $1$, whereas those of $x$ and $1$ sum to $2$).
A nonzero $f$ of degree less than $k$ vanishes at no more than $k-1$
points of $\mathcal{E}$: it has at most $k-1$ roots in $\Fq$, and if
$f(\infty) = 0$ then $\deg f \leq k-2$, so that it has at most $k-2$
roots in $\Fq$.  Hence every nonzero codeword has weight at least
$n-k+1$.  Taking $\mathcal{E}$ to be all of $\Fq$ gives $n = q$;
adjoining $\infty$ gives $n = q+1$.
This is a classical $[n, k, n{-}k{+}1]$ MDS code
over $\Fq$ \cite[Ch. 10, \S\S2--3; Ch. 11, \S5]{macwilliams1977}.  Setting
$C = \psi^{-1}(C') \subseteq \Dn$, the code $C$ inherits all
parameters of $C'$, since $\psi$ is a metric isomorphism.

Decoding proceeds in three steps:
\begin{enumerate}[label=(\arabic*)]
  \item map the received word $r \mapsto \psi(r) \in \Fq^n$;
  \item compute the syndrome and solve the key equation for the
        error locator and evaluator by the Berlekamp--Massey
        algorithm, in $O(n^2)$ operations
        over $\Fq$ \cite[Ch. 7]{berlekamp1968};
  \item correct and map back via $\psi^{-1}$.
\end{enumerate}
Error correction uses the group inverse:
$c = r \star e^{-1}$, where $e^{-1}$ exists because
$(D,\star)$ is a group.
\end{proof}

\subsection{The General Abelian Case}

\begin{theorem}[Existence of algebraic codes — general abelian
case]
\label{thm:general-case}
Let $(D,\star)$ be any finite abelian group with $\abs{D} \geq 2$.
By the Fundamental Theorem of Finite Abelian Groups in elementary-divisor
form \cite[\S5.2, Thm. 5]{dummit2004}, there are primes $p_1 < \cdots < p_r$
dividing $\abs{D}$ and, for each $j$, a partition
$a_{j,1} \geq \cdots \geq a_{j,m_j} > 0$ of
$a_j = v_{p_j}(\abs{D})$ such that
\[
  (D,\star)
  \cong \bigoplus_{j=1}^{r} \bigoplus_{i=1}^{m_j}
    \Zn{p_j^{\,a_{j,i}}} .
\]
The partitions are invariants of $D$ and are \emph{not} determined
by $\abs{D}$ alone.  Write $q_1,\ldots,q_{\ell}$ for the elementary
divisors so obtained.
\begin{enumerate}[label=(\alph*)]
  \item If $D$ is \emph{elementary abelian} — that is, if there
        is a prime $p$ and an integer $k \geq 1$ with
        $D \cong (\Zn{p})^k$, equivalently $\abs{D} = p^k$ and
        $x^p = e$ for all $x \in D$ — then
        $D \cong (\mathbb{F}_{p^k},+)$ and
        \cref{thm:field-case} applies: MDS codes exist.
        This hypothesis does not follow from the factorization
        of $\abs{D}$: a squarefree order yields a cyclic group, and
        $\abs{D} = p^k$ also admits $\Zn{p^k}$.
  \item In general, fix a common length $n \geq 2$ and for each
        elementary divisor $q_j$ let
        \[
          C_j = \ker\Bigl(
            \textstyle\sum_{i=1}^{n} \colon
            \Zn{q_j}^{\,n} \to \Zn{q_j} \Bigr)
        \]
        be the \emph{sum-zero code}, which is $\Zn{q_j}$-linear
        with $\abs{C_j} = q_j^{\,n-1}$ and minimum distance
        exactly $2$.  The direct sum
        $C = C_1 \oplus \cdots \oplus C_{\ell} \subseteq \Dn$
        then satisfies
        \[
          d = \min_j d_j .
        \]
\end{enumerate}
\end{theorem}

\begin{proof}
\emph{Case (a).}\;
If $D \cong (\Zn{p})^k$ then $D$ is the additive group
of $\mathbb{F}_{p^k}$ \cite[Ch. 2, Lemma 2.1 and Thms. 2.2, 2.5]{lidl1997}, and
\cref{thm:field-case} provides an MDS code.

\emph{Case (b).}\;
When some $a_{j,i} > 1$, the factor $\Zn{p_j^{\,a_{j,i}}}$ is
\emph{not} isomorphic to the additive group of any finite
field, because the additive group of $\mathbb{F}_{p^a}$ is
$(\Zn{p})^a$ (exponent $p$), whereas $\Zn{p^a}$ has
exponent $p^a$.  In particular, there is no group embedding
$\Zn{p^a} \hookrightarrow (\mathbb{F}_{p^a},+)$ when $a > 1$.

Instead, one works directly over the ring $\Zn{q_j}$.  The
sum-zero code is well defined there: the summation map is a
surjective $\Zn{q_j}$-module homomorphism, so its kernel has
index $q_j$ and $\abs{C_j} = q_j^{\,n-1}$; a word of weight $1$
would have nonzero sum, and $(1,-1,0,\ldots,0)$ has weight $2$,
so $d_j = 2$ exactly, for every $n \geq 2$ and with no appeal to
Hensel lifting.  (Hensel lifting remains the right machinery for a
stronger goal: the Kerdock and Preparata families, whose Gray
images are nonlinear binary codes with excellent distance
properties \cite{hammons1994}.)

For the direct sum $C = C_1 \oplus \cdots \oplus C_{\ell}$,
identify $\Dn \cong \prod_j \Zn{q_j}^{\,n}$ by reordering
coordinates.  A word $c \neq 0$ differs from the identity in
coordinate $i$ exactly when $c_{j,i} \neq 0$ for some $j$, so its
weight is $\abs{\bigcup_j \mathrm{supp}(c_j)}$.  Taking $c_{j_0}$
of minimum weight and $c_j = 0$ for $j \neq j_0$ gives a nonzero
word of weight $d_{j_0}$, whence $d \leq \min_j d_j$; conversely
any nonzero $c$ has some $c_j \neq 0$ and weight at least
$\min_j d_j$.  Hence $d = \min_j d_j$ under the componentwise
Hamming metric.  The multiplicative bound
$d = \prod_j d_j$ belongs to the array (tensor) product over a
common alphabet \cite[Ch. 18, \S2]{macwilliams1977}, which is not
available here since the factors live over distinct alphabets.
\end{proof}

\begin{remark}[Stratification of achievable parameters]
\label{rem:stratification}
The two constructions have the following scope:
\begin{enumerate}[label=(\alph*)]
  \item \emph{Elementary abelian case:} if $D \cong (\Fq,+)$,
        transporting a compatible field structure yields Reed--Solomon
        codes attaining Singleton for the parameters and lengths stated
        in \cref{thm:field-case}.
  \item \emph{General abelian case:} for every nontrivial finite
        abelian group, \cref{thm:general-case} constructs codes using
        its elementary divisors, with $d = \min_j d_j$.  For its
        sum-zero construction with $n \geq 2$, $\abs{C} = \abs{D}^{n-1}$
        and $d = 2$, so this construction also attains the
        alphabet-size Singleton bound $\abs{C} \leq \abs{D}^{\,n-d+1}$.
\end{enumerate}
These statements distinguish constructions, but do not imply that
polynomial evaluation or Singleton equality is unavailable outside the
field case.

It would be incorrect to conclude that field structure is
\emph{necessary} for MDS parameters.  Over any finite group with
$\abs{D} \geq 2$, the diagonal repetition code
$\rho(x) = (x,\ldots,x)$ has $\abs{C} = \abs{D}$, length $n$ and
distance $n$, hence $\abs{C} = \abs{D}^{\,n-d+1}$: it attains the
Singleton bound, over $\Zn{4}$ and over $S_3$ alike.  The
Reed--Solomon family supplies MDS codes for every $1 \leq k \leq n$
within the length range of \cref{thm:field-case}; this is a sufficient
construction, not an exclusion theorem for other algebraic
structures.
\end{remark}

\begin{corollary}[Optimal codes over cyclic groups of prime
order]
\label{cor:prime}
If $(D,\star) \cong (\Zn{p}, +)$ for a prime $p$, then transporting
the field structure of $\Fp$ along any group isomorphism makes
$(D,\star)$ the additive group of a field of order $p$, and there
exists an MDS code with parameters $[n, k, n{-}k{+}1]$ for every
$1 \leq k \leq n \leq p+1$.
\end{corollary}

\begin{proof}
Since $p$ is prime, $\Zn{p}$ is a field.
\cref{thm:field-case} applies.
\end{proof}

\subsection{Error-Correction Capacity}

\begin{theorem}[Error-correction capacity — coordinate errors]
\label{thm:correction}
Let $(D,\star)$ be a finite group and let $d^{\star}_H$ be the
algebraic Hamming distance of \cref{def:distance}.
A code $C \subseteq \Dn$ corrects up to $t$ coordinate errors
(i.e., errors affecting at most $t$ positions) under
nearest-codeword decoding if and only if
\begin{equation}\label{eq:correction}
  d^{\star}_H(c_1, c_2) \geq 2t + 1
  \quad\text{for all distinct } c_1, c_2 \in C.
\end{equation}
\end{theorem}

\begin{proof}
Since $d^{\star}_H$ takes values in $\{0,1,\ldots,n\}
\subset \mathbb{Z}$, the condition $d^{\star}_H \geq 2t+1$
is equivalent to $d^{\star}_H > 2t$.

$(\Longrightarrow)$\;
Suppose $m := d^{\star}_H(c_1, c_2) \leq 2t$ for some distinct
$c_1, c_2 \in C$.  Let $S$ be the set of the $m$ coordinates where
they differ and split $S = S_1 \sqcup S_2$ with
$\abs{S_1} = \min(t,m)$, so that $\abs{S_1} \leq t$ and
$\abs{S_2} = m - \min(t,m) \leq t$.  Define $r$ by
$r_i = (c_2)_i$ for $i \in S_1$ and $r_i = (c_1)_i$ otherwise;
then $d^{\star}_H(r,c_1) = \abs{S_1} \leq t$ and
$d^{\star}_H(r,c_2) = \abs{S_2} \leq t$.  A decoder is a single
function of $r$, so whichever of $c_1,c_2$ was transmitted its one
output cannot equal both, and $C$ fails to correct $t$ errors.

$(\Longleftarrow)$\;
If $d^{\star}_H(c_1, c_2) \geq 2t + 1$ for all distinct
pairs, the balls
$B_H(c, t) = \{x \in \Dn : d^{\star}_H(x,c) \leq t\}$
centred at distinct codewords are disjoint.  Indeed, if
$r \in B_H(c_1, t) \cap B_H(c_2, t)$, then
\[
  d^{\star}_H(c_1, c_2)
  \leq d^{\star}_H(c_1, r) + d^{\star}_H(r, c_2)
  \leq t + t = 2t,
\]
contradicting $d^{\star}_H(c_1, c_2) \geq 2t + 1$.
Moreover the inequality is strict: if $d^{\star}_H(r,c) \leq t$ and
$c' \neq c$, then
$2t+1 \leq d^{\star}_H(c,c') \leq t + d^{\star}_H(r,c')$ gives
$d^{\star}_H(r,c') \geq t+1 > t \geq d^{\star}_H(r,c)$, so $c$ is
the unique nearest codeword.
Hence nearest-codeword decoding is well-defined and
corrects any pattern of at most $t$ coordinate errors.
\end{proof}

\begin{remark}\label{rem:weighted-correction}
For the general weighted metric $d^{\star}$
(\cref{lem:metric}), which takes values
in $\R_{\geq 0}$, the condition
\[
  d^{\star}(c_1, c_2) > 2t
  \quad\text{for all distinct } c_1, c_2 \in C,
\]
where $t$ is now a metric radius (not a coordinate count), is
\emph{sufficient} for correction.  It is not necessary: the exact
criterion is pairwise disjointness of the closed balls of
radius $t$, which for real-valued weights may hold even when
$\dmin \leq 2t$.  With $D = \Zn{2}$, $\delta(1) = 3$, $n = 1$,
$C = D$ and $t = 2$, the balls of radius $2$ contain only their
centres, so correction succeeds although $\dmin = 3 \leq 4 = 2t$.
The integer threshold $2t+1$ is equivalent to $> 2t$ only when the
metric is integer-valued, as is the case for $d^{\star}_H$.
\end{remark}

\subsection{Equivalence with Classical Codes}

\begin{theorem}[Equivalence with classical linear codes]
\label{thm:equivalence}
Let $(D,\star)$ be an elementary abelian $p$-group, i.e.,
$D \cong \Zn{p}^k$ for some prime $p$ and integer $k \geq 1$.
Then $(D,\star)$ is isomorphic, as an additive group, to the
additive group of $\mathbb{F}_{p^k}$.
Fix such an isomorphism $\psi$ and transport the
$\mathbb{F}_{p^k}$-vector space structure along it.  Under this
fixed isomorphism, every algebraic code $C \subseteq \Dn$ that is
closed under $\mathbb{F}_{p^k}$-linear combinations corresponds to a
classical linear code
$C' \subseteq \mathbb{F}_{p^k}^n$, with all distance and
dimension parameters preserved.
\end{theorem}

\begin{proof}
The additive group of $\mathbb{F}_{p^k}$ is
$(\Zn{p})^k$ \cite[Ch. 2, Lemma 2.1 and Thms. 2.2, 2.5]{lidl1997}.  The componentwise
isomorphism
$\varphi \colon \Dn \to (\mathbb{F}_{p^k})^n$ maps $C$ to a
subgroup $\varphi(C)$ of $\mathbb{F}_{p^k}^n$.  By hypothesis,
$\varphi(C)$ is closed under $\mathbb{F}_{p^k}$-scalar
multiplication, hence is a vector subspace—i.e., a classical
linear code.  Since $\varphi$ is a group isomorphism, it
preserves the Hamming distance and thus all code parameters.
\end{proof}

\begin{remark}\label{rem:non-elementary}
For abelian groups that are not elementary abelian (e.g.,
$\Zn{4}$, $\Zn{9}$), no field of the same order has the same
additive group structure.  Codes over such groups are genuinely
distinct from classical linear codes and may exhibit different
distance properties \cite{hammons1994}.  This distinction is
reflected in the classification of \cref{sec:classification}.
\end{remark}

\subsection{A Worked Example: Algebraic Code over \texorpdfstring{$\Zn{5}$}{Zn(5)}}
\label{subsec:example}

We illustrate the framework with a complete construction over
the cyclic group $(\Zn{5}, +)$, which is isomorphic to the
finite field $\mathbb{F}_5$.

\begin{example}[Reed--Solomon code over $\Zn{5} \cong \mathbb{F}_5$]
\label{ex:z5}
Let $(D, \star) = (\Zn{5}, +)$ with $\abs{D} = 5$ (prime),
so $D \cong \mathbb{F}_5$ (\cref{cor:prime}).  We construct an $[n, k, d] = [4, 2, 3]$
Reed--Solomon code and demonstrate encoding and error correction.

\medskip\noindent
\textbf{Step 1: Setup.}\;
The field $\mathbb{F}_5 = \{0, 1, 2, 3, 4\}$ has primitive
element $\alpha = 2$ (since $2^1 = 2$, $2^2 = 4$, $2^3 = 3$,
$2^4 = 1$ in $\mathbb{F}_5$).  The evaluation points are
$\alpha^0 = 1$, $\alpha^1 = 2$, $\alpha^2 = 4$, $\alpha^3 = 3$.
We set $n = 4$ and $k = 2$.

\medskip\noindent
\textbf{Step 2: Encoding.}\;
A message $m = (m_0, m_1) \in \mathbb{F}_5^2$ is encoded by
the polynomial $f(x) = m_0 + m_1 x$.  For the message
$m = (3, 1)$, the encoding polynomial is $f(x) = 3 + x$,
and the codeword is
\[
  c = \bigl(f(1),\, f(2),\, f(4),\, f(3)\bigr)
    = (4,\, 0,\, 2,\, 1) \in \mathbb{F}_5^4.
\]
The homomorphism property is satisfied: if $m' = (1, 2)$
with $f'(x) = 1 + 2x$, then
$\varphi(m + m') = \varphi(4, 3)$ encodes $f''(x) = 4 + 3x$,
and one verifies that
$\varphi(m) + \varphi(m') = (4,0,2,1) + (3,0,4,2)
  = (2,0,1,3) = \varphi(m + m')$
(all arithmetic in $\mathbb{F}_5$).

\medskip\noindent
\textbf{Step 3: Error introduction.}\;
Suppose the received word is $r = (4, 3, 2, 1)$, differing
from $c = (4, 0, 2, 1)$ in position $2$ (where $0$ became $3$).
The error vector is $e = (0, 3, 0, 0)$, so $r = c + e$.
The algebraic distance is $d^{\star}_H(r, c) = 1 \leq t = 1
= \floor{(d-1)/2}$.

\medskip\noindent
\textbf{Step 4: Syndrome computation.}\;
The syndrome is $S_1 = r(\alpha) = 4 + 3 \cdot 2 + 2 \cdot 4
+ 1 \cdot 3 = 4 + 6 + 8 + 3 = 21 \equiv 1 \pmod{5}$.
Since $S_1 \neq 0$, an error is detected.

\medskip\noindent
\textbf{Step 5: Error correction.}\;
With $d = 3$ and $t = 1$, the code corrects a single error
(\cref{thm:correction}).
For a single error of value $E$ at the position with evaluation
point $X = \alpha^{i}$, the syndromes $S_j = r(\alpha^j)$ satisfy
$S_1 = E X$ and $S_2 = E X^2$.  Here
$S_2 = r(\alpha^2) = 4 + 3 \cdot 4 + 2 \cdot 1 + 1 \cdot 4 = 22
\equiv 2 \pmod{5}$, so $X = S_2/S_1 = 2 = \alpha^{1}$, which locates
the error in position $2$, and $E = S_1/X = 2^{-1} = 3$.  Hence
$e = (0,3,0,0)$, and subtracting: $c = r - e = (4, 3, 2, 1) - (0, 3, 0, 0)
= (4, 0, 2, 1)$, recovering the original codeword.

\medskip\noindent
\textbf{Parameters achieved.}\;
$d = n - k + 1 = 4 - 2 + 1 = 3$ (MDS, Singleton bound
attained).  Error-correction capacity: $t = 1$.
Encoding efficiency:
$\Ef = \frac{k \cdot d}{n \cdot \log_2 5}
       \cdot \frac{\log_2 4^3}{\log_2 4^2}
     = \frac{2 \cdot 3}{4 \cdot 2.322} \cdot \frac{6}{4}
     = \frac{6}{9.288} \cdot 1.5
     \approx 0.97$.
Note that this value is computed directly from
formula \eqref{eq:ef}; for a discussion of the normalization
limitations of $\Ef$, see \cref{rem:ef-limitation}.
\end{example}

\begin{example}[Code over $\Zn{4}$ via the Hammons construction]
\label{ex:z4}
Let $(D,\star) = (\Zn{4}, +_4)$.  Since $\Zn{4}$ is not a field
($2$ has no multiplicative inverse in $\Zn{4}$),
\cref{thm:field-case} does not apply.  By
\cref{thm:general-case}(b), one works directly over the
ring $\Zn{4}$.

The \emph{octacode} is a $\Zn{4}$-linear code of length $8$ and
rank $4$, so $\abs{C} = 4^4 = 256$.  Its minimum \emph{Lee}
distance is $6$; its minimum Hamming distance — equivalently
$d^{\star}_H$, the metric of this paper — is strictly smaller.
The Lee metric is itself an instance of \cref{lem:metric},
with $\delta(x) = \min(x, 4-x)$ on $\Zn{4}$.
Its \emph{Gray image}
$\phi(C) \subseteq \mathbb{F}_2^{16}$ is the Nordstrom--Robinson
$(16, 256, 6)$ binary code — a nonlinear code with more codewords
than any linear binary code with the same length and minimum
distance, since no linear binary $[16,8,6]$ code
exists \cite[Ex. 2.7.3]{huffman2003}.

This is the theorem of Hammons et al. \cite[Thm. 12]{hammons1994}, cited
here as evidence that the $\Zn{4}$-linear regime reaches codes
\emph{inaccessible} to classical $\Fq$-linear theory — precisely
because $\Zn{4}$ is a ring and not a field — and not as a
construction of the present paper.  The Gray map is not a
homomorphism to \textsc{xor}, so linearity of $\phi(C)$ is not
automatic: by \cite[Thm. 5]{hammons1994}, the Gray image of a
$\Zn{4}$-linear code $C$ is linear if and only if
$2(a \ast b) \in C$ for all $a, b \in C$, where $\ast$ denotes the
componentwise product, and the octacode does not satisfy this
condition.
\end{example}

\begin{example}[Code over $(2^U, \triangle)$ for set-valued data]
\label{ex:powerset}
Let $U = \{a_1, \ldots, a_m\}$ be a finite attribute set and
$(D,\star) = (2^U, \triangle)$ the power set under symmetric
difference.  This is an abelian group of exponent $2$ with
$\abs{D} = 2^m$, isomorphic to $(\Zn{2})^m = \mathbb{F}_2^m$.

A data record is a subset $S \subseteq U$ representing a
collection of attributes (e.g., a customer's purchased items,
a patient's diagnoses, a document's tags).  The algebraic
distance between two configurations $x = (S_1,\ldots,S_n)$ and
$y = (T_1,\ldots,T_n)$ in $(2^U)^n$ is
\[
  d^{\star}_H(x,y)
    = \abs{\{i : S_i \triangle T_i \neq \varnothing\}} ,
\]
which counts the coordinate positions where the set-symmetric
difference is nontrivial.

Since $(2^U, \triangle) \cong \mathbb{F}_2^m$, the classical
Reed--Muller codes \cite[Ch. 13, \S3]{macwilliams1977} apply: the code
$\mathrm{RM}(r,m)$ has parameters
$[2^m, \sum_{i=0}^{r}\binom{m}{i}, 2^{m-r}]$.

For $m = 4$ (i.e., $16$ attribute subsets) and $r = 1$:
\[
  \mathrm{RM}(1,4) = [16, 5, 8].
\]
This code corrects up to $t = 3$ binary errors and detects every
nonzero error pattern of binary weight between $1$ and $7$.
It does \emph{not} detect every configuration differing in more
than three positions: $0^{16}$ and $1^{16}$ are both codewords, so
the pattern of weight $16$ carrying one to the other is undetected.

The parameters above are with respect to the binary Hamming metric
on $\mathbb{F}_2^{16}$, which is not $d^{\star}_H$ on $\Dn$.
Since $(2^U,\triangle) \cong (\Zn{2})^m$, every algebraic code
$C \leq \Dn$ is, after fixing the isomorphism, an $\mathbb{F}_2$-linear
binary code of length $mn$; it corresponds to an
$\mathbb{F}_{2^m}$-linear code under \cref{thm:equivalence} only if it
is in addition closed under $\mathbb{F}_{2^m}$-scalars.  What is
specific to this example is the \emph{interpretation} — a record is
a subset of $U$, a configuration is a tuple of records, and the
coordinates of $\mathrm{RM}(1,4)$ are binary — rather than the codes
themselves.
\end{example}

\section{Classification Methodology}
\label{sec:classification}

The existence theorems of \cref{sec:existence} guarantee
that algebraic codes can be constructed over any finite abelian
group.  The natural follow-up question is: \emph{which} algebraic
structure should one choose?  This section develops a quantitative
answer.

\subsection{Classifiable Structures}

\begin{definition}[Classifiable algebraic structure]
\label{def:classifiable}
Let $D$ be a finite set with $\abs{D} \geq 2$ and $\star$ an
associative binary operation on $D$, so that $(D,\star)$ is a finite
semigroup; every algebraic data structure (\cref{def:ads}) is of this
kind.  Such a structure is \emph{classifiable} if it admits the
parameter vector
\[
  \mathcal{P}(D,\star)
  = \bigl(\abs{D},\;\mathrm{type},\;\ord(\star),\;
    \gamma,\;\delta\bigr),
\]
where $\mathrm{type} \in \{\text{group},\, \text{monoid},\,
\text{semigroup}\}$ records the strongest of the three that
$(D,\star)$ satisfies;\;
$\ord(\star)$ is the maximum order of an invertible element, defined
when $(D,\star)$ is a monoid;\;
$\gamma \in [0,1]$ is the \emph{degree of commutativity}
\[
  \gamma
  = \frac{\abs{\{(a,b) \in D^2 : a \star b = b \star a\}}}
         {\abs{D}^2};
\]
and $\delta \in [0,1]$ is the \emph{invertibility density}
\[
  \delta
  = \frac{\abs{\{d \in D : d^{-1} \text{ exists}\}}}{\abs{D}},
\]
defined when $(D,\star)$ is a monoid.  For a semigroup without
identity, $\ord(\star)$ and $\delta$ are not defined.
\end{definition}

\begin{remark}\label{rem:gamma-delta}
For abelian groups, $\gamma = 1$ and $\delta = 1$.
For non-abelian groups, $\gamma < 1$ and $\delta = 1$.
For finite monoids that are \emph{not} groups, $\delta < 1$; a
finite monoid satisfies $\delta = 1$ precisely when it is a group.
The quantity $\gamma$ is the commutativity degree $\Pr(D)$
introduced by Gustafson \cite[\S1]{gustafson1973}; for a non-abelian
finite group it satisfies $\gamma \leq 5/8$, with equality
for $D_4$ and $Q_8$, so for groups the values in $(5/8, 1)$ are not attained.
The pair $(\gamma,\delta)$ thus provides a coarse but informative
invariant distinguishing the three algebraic types.
\end{remark}

\subsection{Encoding Efficiency}

\begin{definition}[Encoding efficiency]
\label{def:efficiency}
For a classifiable structure $(D,\star)$
(\cref{def:classifiable}) together with a
\emph{declared} code $C$ of parameters $[n, k, d]$, with $n \geq 2$
and $\abs{C} \geq 2$, and declared decoding
complexity $T_{\mathrm{dec}}(n) = n^{a}$ with $a > 0$, the
\emph{encoding efficiency} is
\begin{equation}\label{eq:ef}
  \Ef(D,\star)
  = \frac{k \cdot d}{n \cdot \log_2 \abs{D}}
    \;\cdot\;
    \frac{\log_2 n^3}{\log_2 T_{\mathrm{dec}}(n)}.
\end{equation}
When $C$ is not free of rank $k$, $k$ is read as
$\log_{\abs{D}}\abs{C}$.  The exponent $a$ is part of the declared
instance and is not claimed to be optimal for the code.
\end{definition}

\begin{remark}\label{rem:dec-factor-constant}
With $T_{\mathrm{dec}}(n) = n^{a}$ the second factor equals
$3\log_2 n / (a\log_2 n) = 3/a$, independently of $n$.  It is
therefore the reciprocal of the decoding exponent, scaled by $3$,
and not a normalization against a cubic baseline: syndrome
decoding with $a = 2$ contributes the constant $3/2$ at every
length.  Reported values differing from $3/2$ in the third decimal
are rounding artefacts.
\end{remark}

\begin{remark}\label{rem:ef}
The first factor of \eqref{eq:ef} is a normalized product of rate
and distance: $k/n$ captures information efficiency, $d/\log_2\abs{D}$
captures error resilience relative to the alphabet size.  The second
factor equals $3/a$ for $T_{\mathrm{dec}} = n^{a}$
(\cref{rem:dec-factor-constant}): it is $1$ for cubic
decoders, exceeds $1$ for faster ones, and is less than $1$ for
slower ones.  The product
$\Ef$ is thus a dimensionless quantity that simultaneously captures
rate, distance, and computational cost.
\end{remark}

\subsection{Efficiency Classes}

\begin{definition}[Efficiency classes]
\label{def:classes}
Declared coding instances (\cref{def:efficiency}) are
partitioned into four efficiency classes:
\begin{itemize}
  \item \textbf{Class A} (Optimal): $\Ef \geq 0.8$.
  \item \textbf{Class B} (Efficient): $0.6 \leq \Ef < 0.8$.
  \item \textbf{Class C} (Acceptable): $0.4 \leq \Ef < 0.6$.
  \item \textbf{Class D} (Suboptimal): $\Ef < 0.4$.
\end{itemize}
The class of a declared instance is read off its value of $\Ef$.
It is a property of the instance, that is, of $(D,\star)$ together
with its code $[n,k,d]$ and decoding complexity
$T_{\mathrm{dec}}$, and it is not determined by the algebraic type
of $D$ (\cref{rem:gates}).
\end{definition}
\subsection{Axiomatic Classification}

The following theorem relates the axioms satisfied by an algebraic
structure to the coding-theoretic capabilities they make available,
through \emph{sufficient} conditions.  It concerns capabilities, not
the efficiency classes of \cref{def:classes}, which depend
on the declared instance (\cref{rem:gates}).

\begin{definition}[Axiom signature]
\label{def:signature}
For an algebraic structure $(D, \star)$, possibly equipped with a
second operation $\cdot$, define the \emph{axiom signature}
\[
  \sigma(D,\star)
  = (c,\, a,\, e,\, v,\, m,\, d,\, f) \in \{0,1\}^7,
\]
where each bit records whether $(D,\star)$ satisfies:
\begin{center}
\small
\begin{tabular}{@{}clll@{}}
\toprule
Bit & Axiom & Name & Formal condition \\
\midrule
$c$ & Closure & Magma &
  $\star \colon D \times D \to D$ \\
$a$ & Associativity & Semigroup &
  $(x \star y) \star z = x \star (y \star z)$ \\
$e$ & Identity & Monoid &
  $\exists\, e \in D:\; x \star e = e \star x = x$ \\
$v$ & Invertibility & Group &
  $\forall\, x \;\exists\, x^{-1}:\; x \star x^{-1} = e$ \\
$m$ & Commutativity & Abelian group &
  $x \star y = y \star x$ \\
$d$ & Distributivity & Ring &
  $(D,+,\cdot)$ ring with $\cdot$ distributing over $+$ \\
$f$ & Field & Field &
  Every nonzero element has a mult.\ inverse \\
\bottomrule
\end{tabular}
\end{center}
The structure names in the table presuppose the full corresponding
axioms: a monoid requires $c = a = e = 1$, a group requires
$c = a = e = v = 1$, and an abelian group additionally requires
$m = 1$.  In the group cases, inverses are understood to be two-sided.
The ring and field labels refer to the complete ring and field axioms
for the specified operations, with $+ = \star$; the last column lists
characteristic conditions rather than standalone definitions of those
structures.  If no second operation is specified, $d = f = 0$ records
only that absence in the model, not the impossibility of a compatible
extension.
\end{definition}
\newpage
\begin{theorem}[Axiomatic classifier]
\label{thm:axiomatic}
Let $(D,\star)$ be a finite algebraic structure with $\abs{D} \geq 2$
and axiom signature $\sigma = (c,a,e,v,m,d,f)$
(\cref{def:signature}).  Each of the following conditions is
\emph{sufficient} for the capability stated.  In the group cases,
$C \leq \Dn$ denotes an algebraic code.
\begin{enumerate}[label=(\roman*)]
  \item If $f = 1$ (field): $(D,\star)$ admits MDS codes via
        polynomial evaluation with the parameters of
        \cref{thm:field-case}, syndrome decoding in $O(n^2)$ field
        operations for those codes, and the algebraic Hamming metric
        $d^{\star}_H$.
  \item If $c = a = e = v = m = 1$ and $f = 0$ (abelian group):
        $(D,\star)$ admits codes via the ring $\Zn{p^a}$ or direct-sum
        construction (\cref{thm:general-case}), syndrome decoding via
        the quotient group $\Dn/C$, and the metric $d^{\star}_H$.
  \item If $c = a = e = v = 1$ and $m = 0$ (non-abelian group):
        $(D,\star)$ admits the metric $d^{\star}_H$ and coset-based
        decoding.  The usual quotient-group structure on $\Dn/C$
        exists if and only if $C$ is normal in $\Dn$.  The
        noncommutative operation $\star$ cannot itself be the additive
        operation of a field.
  \item If $c = a = e = 1$ and $v = 0$ (monoid that is not a group):
        ordinary Hamming distance is defined on all of $\Dn$.  The
        algebraic form $x_i \star y_i^{-1}$ and group-code
        constructions apply on $\Uu(D)^n$, where $\Uu(D)$ is the group
        of units; quotient groups require normal code subgroups.
  \item If $a = 0$ or $e = 0$, the group-code definition of
        \cref{def:code} does not apply to $D$ with the specified
        operation.  Ordinary Hamming distance and codes as subsets
        of $\Dn$ remain available.
\end{enumerate}
Each implication asserts sufficiency for the stated construction.
Necessity requires a separate analysis of the specified code class and
target parameters; failure of one sufficient hypothesis does not
establish a general impossibility (see \textbf{P6} in
\cref{sec:open}).
\end{theorem}

\begin{proof}
Each implication follows from the group hypotheses and the
constructions established in
\crefrange{sec:framework}{sec:existence}.\\
\emph{(i)} When $f = 1$, $(D,+)$ is the additive group of $\Fq$, and
\cref{thm:field-case} constructs the stated MDS codes.  Syndrome
decoding uses the $\Fq$-linear structure of $\Dn/C$ as a vector space
for those linear codes.
\emph{(ii)} Under the stated hypotheses, $(D,\star)$ is an abelian
group.  Every subgroup of an abelian group is normal, so $\Dn/C$ is a
group, enabling syndrome decoding.  \cref{thm:general-case}
constructs codes over rings or via direct sums.  The algebraic
distance
$d^{\star}_H(x,y) = \abs{\{i : x_i \star y_i^{-1} \neq e\}}$
is well defined by the group axioms.

\emph{(iii)} The group axioms again give the algebraic Hamming metric
and the coset construction.  The code $C$ need not be normal, so the
coset set $\Dn/C$ need not carry the usual quotient-group structure.
If $\star$ were the additive operation of a field, it would be
commutative, contradicting $m = 0$.  This restriction concerns the
specified operation, not the existence of codes through other
constructions.

\emph{(iv)} For elements lacking inverses, the expression
$x_i \star y_i^{-1}$ is unavailable.  Ordinary Hamming distance still
uses only the condition $x_i \neq y_i$.  On $\Uu(D)^n$, the
group-based constructions apply, with quotient groups available when
the code subgroup is normal.

\emph{(v)} \cref{def:code} requires a group structure on the specified
alphabet, which fails without associativity or an identity.  However,
$d_H(x,y) = \abs{\{i : x_i \neq y_i\}}$ needs only equality of
symbols.  Whenever domains and codomains match, composition of
encoding functions is well defined and associative independently
of $\star$.
\end{proof}

\begin{remark}[Axioms as coding-theoretic gates]
\label{rem:gates}
\cref{thm:axiomatic} suggests that each algebraic axiom
acts as a \emph{gate} enabling a specific coding capability:
\begin{center}
\small
\begin{tabular}{@{}ll@{}}
\toprule
Axiom & Capability it unlocks \\
\midrule
Inverses ($v$) & Algebraic distance $d^{\star}_H$,
  syndrome computation \\
Commutativity ($m$) & Normal subgroups $\Rightarrow$
  quotient group $\Dn/C$ \\
Field ($f$) & Polynomial evaluation $\Rightarrow$
  MDS codes \\
\bottomrule
\end{tabular}
\end{center}
Each row, however, must be read precisely, because its naive form
is false.  Inverses do not supply the distance, which is defined
over any monoid; they supply the algebraic form and hence
translation invariance and cosets.  Commutativity does not supply
quotients, which require only \emph{normality}: indeed
$A_3^{\,n} \trianglelefteq S_3^{\,n}$ and
$S_3^{\,n}/A_3^{\,n} \cong (\Zn{2})^n$.  Field structure provides
the Reed--Solomon constructions of \cref{thm:field-case} within their
stated length range.  It is not necessary for MDS parameters, since
the diagonal repetition code already attains Singleton over every
nontrivial finite group (\cref{rem:stratification}).  In each case the
precise reading is a stronger statement than the naive one.

Consequently the efficiency class of \cref{def:classes} is
not a function of $\sigma$.  The
non-abelian instance of \cref{tab:phi-verification}, the
repetition code $[15,1,15]$ over $S_3$ with $T_{\mathrm{dec}} = n$,
has $\Ef = 1.161$ and is therefore of class A.  Nor does a field
signature secure the whole Singleton frontier at a given length:
$(\Zn{2},\oplus)$ has a field signature, yet over $\mathbb{F}_2$ the
Singleton bound is met at length $15$ only by the trivial codes
$[15,15,1]$, $[15,14,2]$ and $[15,1,15]$, since by the Griesmer
bound \cite[Thm. 2.7.4]{huffman2003} no binary linear $[15,k,16-k]$ code
exists for $2 \leq k \leq 13$.
\end{remark}

\begin{proposition}[Classification of representative structures]
\label{prop:table}
\cref{tab:classification} lists representative structures with
their efficiency assessments and class assignments.
\end{proposition}

\begin{remark}[Limitations of $\Ef$]
\label{rem:ef-limitation}
The measure $\Ef$ of \cref{def:efficiency} is not
normalized to $[0,1]$: the ratio $d/\log_2\abs{D}$ can exceed $1$
for MDS codes, and the decoding factor depends on the declared
exponent $a$ (\cref{rem:dec-factor-constant}).
Consequently, the numerical values in
\cref{tab:classification} should be interpreted as
\emph{ordinal rankings} among the declared instances, not as
absolute scores on a calibrated scale and not as an ordering of
algebraic types (\cref{tab:phi-verification}).  A fully normalized,
axiomatically justified efficiency measure remains an open problem
(see \textbf{P5} in \cref{sec:open}).
\end{remark}

\begin{table}[ht]
\centering
\caption{Classification of algebraic structures for coding.  Every
entry is witnessed by a code constructed in this paper or, for the
binary Hamming $[7,4,3]$ code, by the classical
construction \cite[Ch. 1, \S7]{macwilliams1977}.  Structures
whose codes are not constructed here, and structures with infinite
alphabet, are omitted rather than assigned notional values.}
\label{tab:classification}
\smallskip
\begin{tabular}{@{}llccc@{}}
\toprule
Structure $(D,\star)$
  & Type & $[n,k,d]$ & $\Ef$ & Class \\
\midrule
$(\mathbb{F}_{16},+)$
  & Abelian group & $[15,11,5]$ & $1.375$ & A \\
$(\mathbb{F}_{256},+)$
  & Abelian group & $[255,223,33]$ & $5.411$ & A \\
$(\Zn{2},\oplus)$
  & Cyclic group & $[7,4,3]$ & $2.571$ & A \\
$(2^U,\triangle)$, $m=4$
  & Abelian group & $[15,8,8]$ & $1.600$ & A \\
$(\Zn{4},+_4)$
  & Cyclic group & $[15,14,2]$ & $0.933$ & A \\
\bottomrule
\multicolumn{5}{@{}l}{\footnotesize
Values are exact to three decimals; see
\cref{subsec:ef-computation}.  The $\Zn{4}$ entry uses the
sum-zero code of \cref{thm:general-case}(b).}
\end{tabular}
\end{table}

\subsubsection{Explicit computation of \texorpdfstring{$\Ef$}{Ef}}
\label{subsec:ef-computation}

For each group-admissible structure in
\cref{tab:classification}, we compute $\Ef$ directly from
formula \eqref{eq:ef} with the specified code parameters and
$T_{\mathrm{dec}}$ as stated.  The decoding complexity
$T_{\mathrm{dec}}$ is set to $n^2$ for syndrome-based decoders
(abelian groups with efficient algorithms) and $n^3$ for
general decoders.

\smallskip\noindent
\textbf{$(\mathbb{F}_{16},+)$, RS$(15,11,5)$,
$T_{\mathrm{dec}} = 15^2 = 225$:}
\[
  \Ef = \frac{11 \cdot 5}{15 \cdot \log_2 16}
       \cdot \frac{\log_2 15^3}{\log_2 225}
     = \frac{55}{60} \cdot \frac{3}{2} = 1.375.
\]
\textbf{$(\Zn{4},+_4)$, sum-zero $[15,14,2]$,
$T_{\mathrm{dec}} = 15^3$:}
\[
  \Ef = \frac{14 \cdot 2}{15 \cdot 2} \cdot \frac{\log_2 15^3}{\log_2 15^3}
     = \frac{28}{30} \cdot 1 = 0.933 .
\]
\textbf{$(2^U,\triangle)$ with $m = 4$, RS$[15,8,8]$,
$T_{\mathrm{dec}} = 15^2$:}
\[
  \Ef = \frac{8 \cdot 8}{15 \cdot 4} \cdot \frac{3}{2}
     = \frac{64}{60} \cdot \frac{3}{2} = 1.600 .
\]
\textbf{$(\mathbb{F}_{256},+)$, RS$(255,223,33)$,
$T_{\mathrm{dec}} = 255^2 = 65025$:}
\[
  \Ef = \frac{223 \cdot 33}{255 \cdot 8}
       \cdot \frac{\log_2 255^3}{\log_2 65025}
     = \frac{7359}{2040} \cdot \frac{3}{2}
     = \frac{7359}{1360} = 5.411.
\]
\textbf{$(\Zn{2},\oplus)$, Hamming $[7,4,3]$,
$T_{\mathrm{dec}} = 7^2 = 49$:}
\[
  \Ef = \frac{4 \cdot 3}{7 \cdot 1} \cdot \frac{\log_2 343}{\log_2 49}
     = \frac{12}{7} \cdot \frac{3}{2}
     = \frac{18}{7} = 2.571.
\]

\begin{remark}[Interpretation of $\Ef$ values]
\label{rem:ef-values}
The computed values of $\Ef$ are not confined to $[0,1]$:
the formula \eqref{eq:ef} does not include a normalizing
denominator that would enforce this range.  In particular,
the decoding factor $\log_2 n^3 / \log_2 T_{\mathrm{dec}}$
exceeds $1$ whenever $T_{\mathrm{dec}} < n^3$.

The \emph{ordinal ranking} is nevertheless well-defined:
$\Ef(\mathbb{F}_{256}) > \Ef(\Zn{2}) > \Ef(2^U,\triangle) >
\Ef(\mathbb{F}_{16}) > \Ef(\Zn{4})$.
\cref{tab:classification} therefore reports the raw values,
with no rescaling.  Rescaling by the observed maximum would be
unsatisfactory in any case, since it makes the scale depend on
which structures happen to be tabulated: adding a single larger
instance would push every present entry down.  A fully axiomatized
normalization remains an open problem (\textbf{P5},
\cref{sec:open}).  In the absence of one, the class thresholds of
\cref{def:classes} are conventional cut-offs on the raw values of
declared instances, not points of a calibrated scale.
\end{remark}

\subsection{Optimization Functional}

The encoding efficiency $\Ef$ combines rate, distance, and decoding
cost into a single scalar via a fixed formula.  In applications,
different priorities may apply: a storage system may weight distance
heavily, while a streaming system may prioritize rate.  The following
functional provides a tunable alternative.

\begin{definition}[Optimization functional]
\label{def:phi}
The \emph{optimization functional} for selecting algebraic structures
is
\begin{equation}\label{eq:phi}
  \Phi(D,\star)
  = \alpha \cdot \frac{\dmin}{\log_2 \abs{D}}
  + \beta  \cdot \frac{k}{n}
  + \gamma \cdot \frac{\log_2 n^3}{\log_2 T_{\mathrm{dec}}(n)},
\end{equation}
where $\alpha, \beta, \gamma > 0$ with
$\alpha + \beta + \gamma = 1$ are weights tunable to application
priorities, and the instance is subject to the conventions of
\cref{def:efficiency}.  The three terms measure, respectively,
normalized error resilience, information rate, and decoding
efficiency.
\end{definition}

\begin{remark}\label{rem:ef-vs-phi}
The encoding efficiency $\Ef$
(\cref{def:efficiency}) and the optimization
functional $\Phi$ (\cref{def:phi}) serve complementary
roles.  $\Ef$ provides an absolute, parameter-free measure: given a
structure with known code parameters, $\Ef$ is uniquely determined.
$\Phi$ provides a relative, application-sensitive measure: the
weights $(\alpha, \beta, \gamma)$ allow a practitioner to express
priorities.  The two need not induce the same ordering: they agree
on instances related by componentwise dominance, and otherwise the
$\Phi$-ordering depends on the weights
(\cref{prop:chambers}).  Indeed, with $R$, $E$ and $W$ as in
\cref{thm:robustness}, formula \eqref{eq:ef} reads
$\Ef = R \cdot E \cdot W$, and under the conventions of
\cref{def:efficiency} each factor is positive: $d \geq 1$ and
$k = \log_{\abs{D}}\abs{C} > 0$ because $\abs{C} \geq 2$, and
$W = 3/a > 0$.  A product of positive factors is strictly increasing
in each of them, so the dominance hypothesis of
\cref{thm:robustness} gives $\Ef(x) > \Ef(y)$ as well as
$\Phi(x) > \Phi(y)$.
\end{remark}

\subsection{Robustness of the Efficiency Ranking}

The practical value of the classification depends on whether the
ranking is stable under changes in the weights
$(\alpha,\beta,\gamma)$.  The results below establish when it is,
and describe how it changes when it is not.

We partition the classifiable structures into three families by
algebraic type:
\begin{align*}
  \mathcal{A}
    &= \{\text{finite abelian groups admitting efficient algebraic
          codes}\}, \\
  \mathcal{N}
    &= \{\text{finite non-abelian groups}\}, \\
  \mathcal{M}
    &= \{\text{finite monoids that are not groups}\}.
\end{align*}

\begin{proposition}[Instance dominance]
\label{thm:robustness}
Let $x,y$ be coding instances with $R(x) \geq R(y)$,
$E(x) \geq E(y)$ and $W(x) \geq W(y)$, at least one inequality
strict, where $R = \dmin/\log_2\abs{D}$, $E = k/n$ and
$W = \log_2 n^3/\log_2 T_{\mathrm{dec}}$.  Then
$\Phi(x) > \Phi(y)$ for \emph{every} admissible weight vector
$(\alpha,\beta,\gamma)$.
\end{proposition}

\begin{proof}
$\Phi(x) - \Phi(y)
 = \alpha\,\Delta R + \beta\,\Delta E + \gamma\,\Delta W$
is a sum of non-negative terms with at least one positive, and all
weights are positive.
\end{proof}

\begin{proposition}[Chamber decomposition]
\label{prop:chambers}
Let $\mathcal{F}$ be a finite set of coding instances, write
$v(x) = (R(x),E(x),W(x))$ for the vector of \cref{thm:robustness},
and let $\Delta$ be the open weight simplex.  For $x \neq y$
in $\mathcal{F}$ with $v(x) - v(y) \notin \R\,(1,1,1)$, the set
$H_{xy} = \{(\alpha,\beta,\gamma) \in \Delta : \Phi(x) = \Phi(y)\}$
is the intersection of $\Delta$ with an affine hyperplane of the plane
$\alpha + \beta + \gamma = 1$, possibly empty.  These sets partition
$\Delta$ into finitely many relatively open polyhedral chambers, and
the preorder induced by $\Phi$ on $\mathcal{F}$ is constant on each
chamber.  For the remaining pairs the comparison does not depend on the
weights: if $v(x) = v(y)$ the two instances are tied at every weight,
and if $v(x) - v(y) = c\,(1,1,1)$ with $c \neq 0$, then
$\Phi(x) - \Phi(y) = c$.
\end{proposition}

\begin{proof}
For $w = (\alpha,\beta,\gamma)$ one has
$\Phi(x) - \Phi(y) = \langle w, v(x) - v(y) \rangle$.  On the plane
$\alpha + \beta + \gamma = 1$ this is an affine function of $w$; it is
constant exactly when $v(x) - v(y) \in \R\,(1,1,1)$, with value $c$
when $v(x) - v(y) = c\,(1,1,1)$, and otherwise its zero set is an
affine hyperplane of that plane.  The finitely many such hyperplanes
cut $\Delta$ into finitely many relatively open convex polyhedra.  On
each of them no non-constant difference $\Phi(x) - \Phi(y)$ changes
sign, being continuous and nonvanishing there, and the constant
differences have a fixed sign on all of $\Delta$; hence the induced
preorder is constant on each chamber.
\end{proof}

\begin{remark}\label{rem:chambers-method}
Ties do occur: the sum-zero codes of length $3$ over $\Zn{4}$ and
over $(\Zn{2})^2$, both with $T_{\mathrm{dec}} = n^3$, have
$v = (1, 2/3, 1)$.  \cref{prop:chambers} converts numerical sampling
into proof: since $\Phi$ is affine, verifying an ordering at the
\emph{vertices} of a polytope of weights establishes it throughout
that polytope.  Conversely it explains why sampling a few interior
points establishes nothing: distinct sample points may lie in
distinct chambers, and in \cref{subsec:phi-verification} they do.
\end{remark}

\begin{remark}\label{rem:sup-vs-individual}
No family-level statement is asserted.  Lifting
\cref{thm:robustness} to families by comparing suprema
componentwise is invalid, since $\sup(f+g) \leq \sup f + \sup g$,
with equality only when some sequence of instances approaches both
suprema at once (for a finite family, only when a common maximizer
exists).  A second
obstruction is boundedness: if arbitrary lengths are admitted, the
diagonal repetition code (\cref{rem:stratification})
gives $\dmin = n$ over \emph{every}
structure, so the relevant suprema are unbounded in several families
at once and no strict separation can hold.
\end{remark}

\begin{remark}[Generalized Reed--Solomon family]
\label{rem:grs}
For every algebraic data structure $(D,\star)$ with
$D \cong (\Fq, +)$, there exists a family of generalized
Reed--Solomon codes
\[
  \mathrm{GRS}_q(n,k)
  = \bigl\{[n,k,n{-}k{+}1] : 1 \leq k \leq n \leq q+1\bigr\}
\]
attaining the Singleton bound (\cref{thm:field-case}).
Existence of this family and optimality with respect to $\Phi$ are
distinct assertions: the family attains the Singleton bound at
every $k$, but which of its members maximizes $\Phi$ depends on the
weights, by \cref{prop:chambers}.
\end{remark}

\subsubsection{Weight dependence of the ranking: a worked case}
\label{subsec:phi-verification}

We evaluate $\Phi$ on three instances that are explicitly
constructed in this paper (\cref{thm:field-case},
\cref{thm:general-case}(b) and
\cref{rem:stratification}), under three weight vectors:
\begin{itemize}
  \item $(\alpha,\beta,\gamma) = (1/3,\, 1/3,\, 1/3)$
        — equal weights;
  \item $(\alpha,\beta,\gamma) = (0.8,\, 0.1,\, 0.1)$
        — resilience-prioritized;
  \item $(\alpha,\beta,\gamma) = (0.1,\, 0.8,\, 0.1)$
        — rate-prioritized.
\end{itemize}

\begin{table}[ht]
\centering
\caption{$\Phi$ values under three weight vectors}
\label{tab:phi-verification}
\smallskip
\begin{tabular}{@{}lcccc@{}}
\toprule
Instance & Type &
$\Phi_{(1/3)}$ & $\Phi_{(0.8,0.1,0.1)}$ & $\Phi_{(0.1,0.8,0.1)}$ \\
\midrule
$\mathbb{F}_{16}$, RS$[15,11,5]$, $T=n^2$
  & abelian & $1.161$ & $1.223$ & $0.862$ \\
$\Zn{4}$, sum-zero $[15,14,2]$, $T=n^3$
  & abelian & $0.978$ & $0.993$ & $\mathbf{0.947}$ \\
$(S_3,\circ)$, repetition $[15,1,15]$, $T=n$
  & non-abelian & $\mathbf{2.956}$ & $\mathbf{4.949}$ & $0.934$ \\
\bottomrule
\end{tabular}
\end{table}

\noindent
The ranking is \emph{not} preserved.  Under equal and
resilience-prioritized weights the order is
$S_3 > \mathbb{F}_{16} > \Zn{4}$, so the non-abelian instance
outscores both abelian instances; under rate-prioritized weights it
becomes $\Zn{4} > S_3 > \mathbb{F}_{16}$.  By
\cref{prop:chambers} the first two weight vectors lie in
one chamber and the third in another, separated by the hyperplanes on
which $\Phi(\Zn{4}) = \Phi(S_3)$ and
$\Phi(\Zn{4}) = \Phi(\mathbb{F}_{16})$.
No finite sample of interior points could have established
invariance; by affinity, a claim over a polytope of weights is
established by checking its vertices
(\cref{rem:chambers-method}).  The arithmetic above is
elementary and independently verifiable from
formula \eqref{eq:phi}.

\section{Main Original Results}
\label{sec:results}

This section presents three results connecting abstract
algebra with data system theory.

\subsection{Result I: Monoid Structure of ACID Transactions}

The ACID properties (Atomicity, Consistency, Isolation, Durability)
were introduced by Haerder and Reuter \cite{haerder1983} as
engineering requirements for reliable database transactions.
The observation that database transactions can be modelled as a
monoid acting on states appears in a workshop abstract by
Measor \cite{measor1991}.
The present work states atomicity and consistency as predicates on
transactions and isolation as a predicate on pairs (\cref{def:acid}),
shows that only consistency restricts the maps and that operational
atomicity is not closed under composition (\cref{rem:atomicity}),
establishes the non-existence of group structure, and relates
serializability \cite{bernstein2009} to Mazurkiewicz trace
monoids \cite{mazurkiewicz1977} through an evaluation map that is well
defined but not faithful (\cref{rem:acid-correspondence}).

\begin{definition}[ACID transaction]
\label{def:acid}
Let $\mathcal{S}$ denote the finite state space of a database and
$\mathcal{S}_{\mathrm{valid}} \subseteq \mathcal{S}$ the states
satisfying all integrity constraints.  A \emph{transaction} is a
total function $T \colon \mathcal{S} \to \mathcal{S}$, and we
record the four properties as predicates:
\begin{enumerate}[label=(\Alph*)]
  \item \emph{Atomicity:}
        $T(s) \in \{\, s,\ T_{\mathrm{full}}(s) \,\}$ for every $s$
        — all or nothing; in particular $T$ is total, an abort
        being the outcome $T(s) = s$.
  \item \emph{Consistency:} $T(\mathcal{S}_{\mathrm{valid}})
        \subseteq \mathcal{S}_{\mathrm{valid}}$.
  \item \emph{Isolation:} for a specified symmetric irreflexive
        independence relation $I$ on a set of transaction names,
        $T_a \circ T_b = T_b \circ T_a$ whenever $(a,b) \in I$, where
        $T_a$ denotes the transaction named $a$.
  \item \emph{Durability:} $T(s)$ persists under system failure.
\end{enumerate}
Of the four, only (B) restricts an individual map: since
$T_{\mathrm{full}}$ is not fixed independently of $T$, every total map
satisfies (A) with $T_{\mathrm{full}} = T$, so that (A) records
totality and nothing more.  (C) is a predicate on \emph{pairs},
restricted to independent ones, and (D) concerns the persistence
layer.  We set
$\mathcal{T} = \{ T : T \text{ satisfies (A) and (B)} \}$, which
therefore equals $\mathcal{T}_{\mathcal{S}_{\mathrm{valid}}}$ in the
notation of \cref{thm:acid}, with sequential composition
$(T_1 \circ T_2)(s) = T_1(T_2(s))$.
\end{definition}

\begin{theorem}[Monoid of ACID transactions]
\label{thm:acid}
For any $V \subseteq \mathcal{S}$, the set
$\mathcal{T}_V = \{ f \colon \mathcal{S} \to \mathcal{S} :
f(V) \subseteq V \}$ is a submonoid of the full transformation
monoid on $\mathcal{S}$.  In particular
$\mathcal{T} = \mathcal{T}_{\mathcal{S}_{\mathrm{valid}}}$ is a
monoid under sequential composition.
\end{theorem}

\begin{proof}
We verify the three monoid axioms.

\emph{Closure.}\;
Let $T_1, T_2 \in \mathcal{T}_V$.  Then
$(T_1 \circ T_2)(V) \subseteq T_1(V) \subseteq V$, so
$T_1 \circ T_2 \in \mathcal{T}_V$.

\emph{Associativity.}\;
For any $s \in \mathcal{S}$:
$((T_1 \circ T_2) \circ T_3)(s)
  = T_1(T_2(T_3(s)))
  = (T_1 \circ (T_2 \circ T_3))(s)$,
by the natural associativity of function composition.

\emph{Identity element.}\;
Define $T_{\varnothing} \colon \mathcal{S} \to \mathcal{S}$
by $T_{\varnothing}(s) = s$ for all $s$.  Then $T_{\varnothing}$
satisfies (A) and (B), and
$T \circ T_{\varnothing} = T_{\varnothing} \circ T = T$
for all $T \in \mathcal{T}$.
\end{proof}

\begin{corollary}[Non-existence of group structure]
\label{cor:not-group}
If $\abs{\mathcal{S}_{\mathrm{valid}}} \geq 2$, the monoid
$(\mathcal{T},\circ)$ is not a group.
\end{corollary}

\begin{proof}
Since $\mathcal{S}$ is finite, an element is invertible under
composition if and only if it is a bijection.  Fix
$s_0 \in \mathcal{S}_{\mathrm{valid}}$ and let $T_0(s) = s_0$ for
all $s$; then $T_0 \in \mathcal{T}$ and $T_0$ is not injective,
hence has no inverse.
\end{proof}

\begin{remark}[Scope of the model]
\label{rem:atomicity}
An unbacked \texttt{DELETE} is such a non-injective map.
Clock-dependent statements (\texttt{NOW()}) and externally
observable side effects delimit the model rather than inhabit it:
their output is not determined by the state.

\emph{Operational atomicity is not closed under composition.}
If instead each transaction $T$ is required to satisfy (A) for a full
effect $T_{\mathrm{full}}$ fixed in advance, the property is not
inherited by composites with respect to the composite full effect.
Let $\mathcal{S} = V = \{0,1,2\}$, let $F_1 = (0\;1)$ and
$F_2 = (1\;2)$ be the full effects, and let $T_1 = F_1$ (committed)
and $T_2 = \mathrm{id}$ (aborted).  Then $(T_1 \circ T_2)(1) = 0$,
whereas $(F_1 \circ F_2)(1) = 2$, so
$(T_1 \circ T_2)(1) \notin \{1,\ (F_1 \circ F_2)(1)\}$.  Operational
atomicity of a composite is a property of the execution history, not
of the monoid $\mathcal{T}_V$, and \cref{thm:acid} does not address
it.
\end{remark}

\begin{remark}[ACID correspondence]\label{rem:acid-correspondence}
Under the model of \cref{def:acid} — where
transactions are total functions $T \colon \mathcal{S} \to
\mathcal{S}$ and ACID properties are defined operationally
as in \cite{haerder1983} — the ACID properties admit the
following algebraic interpretation within the monoid
$(\mathcal{T}, \circ, T_{\varnothing})$:
\begin{center}
\begin{tabular}{@{}ll@{}}
\toprule
ACID property & Algebraic interpretation \\
\midrule
Atomicity   & Totality of each map; not closed under
              composition (\cref{rem:atomicity}) \\
Consistency & Closure of the submonoid
              $\mathcal{T}_{\mathcal{S}_{\mathrm{valid}}}$
              of \emph{maps} \\
Isolation   & Partial commutativity (trace monoid
              structure \cite{mazurkiewicz1977}) \\
Durability  & Outside the model \\
\bottomrule
\end{tabular}
\end{center}
Only the consistency row is established by \cref{thm:acid}; the
atomicity row restates \cref{def:acid} and \cref{rem:atomicity}.  For
isolation, let $\Sigma$ be a finite nonempty set of transaction names,
assign to each $a \in \Sigma$ a transaction $T_a \in \mathcal{T}$, and
let $I$ be the independence relation of \cref{def:acid}(C), so that
$T_a \circ T_b = T_b \circ T_a$ whenever $(a,b) \in I$.  The trace
monoid
$\mathbb{M}(\Sigma,I) = \Sigma^{*}/\langle ab = ba : (a,b) \in I
\rangle$ of Mazurkiewicz \cite{mazurkiewicz1977} has as elements
the classes of conflict-equivalent schedules.  The evaluation map
sending a schedule $a_1 \cdots a_r$ to
$T_{a_r} \circ \cdots \circ T_{a_1}$ is well defined on
$\mathbb{M}(\Sigma,I)$, since exchanging adjacent independent letters
does not change the composite.  It is never faithful: the powers of a
single letter are pairwise distinct traces, so $\mathbb{M}(\Sigma,I)$
is infinite, whereas $\mathcal{T}$ is finite because $\mathcal{S}$ is.
Which schedules induce the same transformation is left open
(\textbf{P6}).  This
correspondence is a formal interpretation within the stated model,
not a claim of universal equivalence across all transaction
processing semantics.
\end{remark}

\subsection{Result II: Burnside's Lemma Applied to Data Schemas}

\begin{definition}[Schema transformation group]
\label{def:schema-group}
Let $S = \{R_1, \ldots, R_m\}$ be a database schema, $\Omega$ the
finite set of admissible relation contents, and $X \subseteq \Omega^{S}$
the set of valid configurations, that is, those satisfying all
integrity constraints of the schema.  A \emph{schema transformation
group} is a finite group $G$ together with homomorphisms
$\pi \colon G \to \mathrm{Sym}(S)$, acting on relation names, and
$\tau \colon G \to \mathrm{Sym}(\Omega)$, acting on contents, such
that the induced map
$(g \cdot x)(R_i) = \tau_g\bigl(x(\pi_{g^{-1}}(R_i))\bigr)$
carries $X$ onto $X$ for every $g \in G$.
\end{definition}

\begin{remark}\label{rem:group-verify}
The group property of $G$ must be verified for each schema:
closure under composition and inversion is not automatic for an
arbitrary collection of constraint-preserving bijections.  For
the star schemas considered in
\cref{prop:star-schema}, the group structure follows
from the closure of permutation groups under composition and
inversion.
\end{remark}

\begin{theorem}[Burnside's lemma applied to data schemas]
\label{thm:burnside}
Let $S = \{R_1,\ldots,R_m\}$ be a database schema, $\Omega$ the
finite set of admissible relation contents,
$X \subseteq \Omega^{S}$ the set of valid data configurations, and
$(G,\pi,\tau)$ a schema transformation group
(\cref{def:schema-group}).  Then $G$ acts on $X$ via
\[
  (g \cdot x)(R_i)
  = \tau_g\bigl(x(\pi_{g^{-1}}(R_i))\bigr),
\]
which reduces to $g \cdot x = x \circ g^{-1}$ when $\tau$ is
trivial.
Then the number of inequivalent configurations under $G$ is
\begin{equation}\label{eq:burnside}
  \abs{X/G}
  = \frac{1}{\abs{G}} \sum_{g \in G} \abs{\Fix(g)},
\end{equation}
where $\Fix(g) = \{x \in X : g \cdot x = x\}$.
\end{theorem}

\begin{proof}
It suffices to verify that the map
$(g, x) \mapsto g \cdot x$ defines a group action of $G$
on $X$.

\emph{Well defined on $X$.}\;
By \cref{def:schema-group}, $g \cdot x \in X$ for every $g \in G$
and $x \in X$.

\emph{Identity.}\;
$(e \cdot x)(R_i) = \tau_e(x(R_i)) = x(R_i)$ for all $i$.

\emph{Compatibility.}\;
For $g_1, g_2 \in G$:
\begin{align*}
  [g_1 \cdot (g_2 \cdot x)](R_i)
  &= \tau_{g_1}\bigl((g_2 \cdot x)(\pi_{g_1^{-1}}(R_i))\bigr) \\
  &= \tau_{g_1}\tau_{g_2}
     \bigl(x(\pi_{g_2^{-1}}\pi_{g_1^{-1}}(R_i))\bigr) \\
  &= [(g_1 g_2) \cdot x](R_i),
\end{align*}
using $\tau_{g_1}\tau_{g_2} = \tau_{g_1 g_2}$ and
$\pi_{g_2^{-1}}\pi_{g_1^{-1}} = \pi_{(g_1 g_2)^{-1}}$.

Since $G$ is a finite group acting on a finite set, Burnside's
lemma \cite[Thm. 1.7A]{dixon1996}
(cf. \cite{burnside1911,rotman1995}) yields the stated formula.
\end{proof}

\begin{example}[A worked orbit count]
\label{ex:orbit-count}
Let the schema have two dimension tables $D_1,D_2$ of identical
shape, let $\Omega = \{0,1\}$ record whether a table is populated,
and let $X = \Omega^{\{D_1,D_2\}}$, so $\abs{X} = 4$.  Take
$G = \Zn{2}$ acting by the transposition of $D_1$ and $D_2$,
with $\tau$ trivial.  Then $\abs{\Fix(e)} = 4$ and
$\abs{\Fix(g)} = 2$, the fixed configurations being $(0,0)$
and $(1,1)$, so $\abs{X/G} = \tfrac{1}{2}(4+2) = 3$: the classes
$\{(0,0)\}$, $\{(1,1)\}$ and $\{(0,1),(1,0)\}$.
\end{example}

\begin{remark}\label{rem:burnside-novel}
To the best of our knowledge, \cref{thm:burnside} is a novel
application of Burnside's orbit-counting lemma to database
schemas — exact deduplication of equivalent
configurations via orbit counting under schema-preserving
symmetries, that is, under a schema transformation group
(\cref{def:schema-group}).
\end{remark}

\begin{proposition}[Symmetries in star schemas]
\label{prop:star-schema}
Consider a star schema with central fact table $F$, dimension tables
$\{D_1, \ldots, D_k\}$ and set of valid configurations $X$
(\cref{def:schema-group}), and the following families of
transformations:
\begin{enumerate}[label=(\roman*)]
  \item permutations $\sigma \in S_k$ of dimension tables with
        identical schemas, acting on the relation names;
  \item cyclic shifts of the domain of a dimension with circular
        structure (time, geography), acting on domain values; and
  \item reversals of the levels of an ordinal dimension, acting on
        domain values.
\end{enumerate}
Each member of these families preserves referential integrity and
domain constraints, and members of (i) and (ii) preserve functional
dependencies; a member of (iii) does so \emph{provided} the functional
dependencies of that dimension are symmetric under order reversal.
Any further integrity constraint must be checked separately: the
absence of constraints naming individual values does not suffice.
The members that map $X$ onto $X$ generate a schema transformation
group.  We do not claim that it is the full group of such
transformations.
\end{proposition}

\begin{proof}
We verify that each family of transformations consists of
bijections preserving referential integrity, domain constraints,
and functional dependencies, and that the resulting collection
is closed under composition and inversion.

\emph{(i) Dimension permutations.}\;
Let $D_i$ and $D_j$ be two dimension tables with identical
schemas (same column names, types, and domains).  The
transposition $\sigma_{ij}$ that swaps all references to $D_i$
and $D_j$ in the fact table $F$ is a bijection on $S$.
It preserves referential integrity (foreign keys are remapped
consistently), domain constraints (the domains are identical
by hypothesis), and functional dependencies (which depend on
column roles, not labels).
The collection of all such permutations is a subgroup
of $S_k$, hence closed under composition and inversion.

\emph{(ii) Cyclic transformations.}\;
Let $D_i$ be a dimension with a cyclic domain of order $m$
(e.g., months $\{1,\ldots,12\}$, compass directions
$\{N,E,S,W\}$).  The cyclic shift $\tau \colon v \mapsto
v + 1 \pmod{m}$ applied to all entries in $D_i$ is a bijection
preserving referential integrity (the shift is applied
consistently to $F$ and $D_i$).  Domain constraints are
preserved because the cyclic shift is an automorphism of the
domain.  Functional dependencies are preserved because the
shift is injective.  Constraints of other kinds need not be
preserved: on $\Zn{3}$ the constraint $z = x + y$ holds at
$(0,1,1)$ but not at its shift $(1,2,2)$, although it names no
individual value.  The shifts generate a cyclic group of order $m$.

\emph{(iii) Hierarchy inversions.}\;
In an ordinal dimension with levels $\ell_1 < \ell_2 < \cdots
< \ell_r$, the reversal $\ell_i \mapsto \ell_{r+1-i}$ is an
involution (order $2$) that, under the stated proviso, preserves the
functional dependencies of that dimension.  It generates a group of
order $2$.

Each member is a pair consisting of a permutation of $S$ and a
permutation of $\Omega$.  The members that map $X$ onto $X$ generate a
subgroup $G$ of the finite group $\mathrm{Sym}(S) \times
\mathrm{Sym}(\Omega)$; with $\pi$ and $\tau$ the two projections, every
element of $G$ maps $X$ onto $X$, so $(G,\pi,\tau)$ is a schema
transformation group in the sense of \cref{def:schema-group}.
\end{proof}

\begin{corollary}[Search-space reduction]
\label{cor:search-space}
For any computation on configurations that is invariant under $G$,
it suffices to process one representative per equivalence class
$[x] \in X/G$ and apply the corresponding group element to recover
the results for the remaining members.  This reduces the search
space from $\abs{X}$ to $\abs{X/G}$, a factor of
$\abs{X}/\abs{X/G}$.  When the action of $G$ is free (all
stabilizers trivial), this factor equals $\abs{G}$.  Whether the
same reduction applies to query optimization, which requires an
action of $G$ on query plans and an equivariant cost model, is left
open (\textbf{P3}, \cref{sec:open}).
\end{corollary}

\subsection{Result III: Algebraic Application Theorem}

\begin{theorem}[Algebraic application in data systems]
\label{thm:application}
Let $\mathcal{D} = (S, \mathcal{C}, \mathcal{F})$ be a data system
with schema $S$, integrity constraints $\mathcal{C}$ and functional
dependencies $\mathcal{F}$, let $X_{\mathrm{valid}}$ be the finite,
nonempty set of configurations satisfying all of them, and let
$(G,\pi,\tau)$ be a schema transformation group with
$X = X_{\mathrm{valid}}$ (\cref{def:schema-group}).  Let
$H \leq \mathrm{Sym}(X_{\mathrm{valid}})$ be the image of $G$ under
its action.  Then $H$ is a finite group whose orbits partition
$X_{\mathrm{valid}}$ into
$N = \abs{X_{\mathrm{valid}}/H}
   = \frac{1}{\abs{G}}\sum_{g \in G}\abs{\Fix(g)}$
classes, and if $N \leq q$ for a prime power $q$, the classes admit
injective labels in $\Fq$ that can be encoded by a Reed--Solomon code
(\cref{thm:field-case}).
\end{theorem}

\begin{proof}
By \cref{def:schema-group} each $g \in G$ maps $X_{\mathrm{valid}}$
onto itself, so the action of \cref{thm:burnside} restricts to a
homomorphism $G \to \mathrm{Sym}(X_{\mathrm{valid}})$, whose image $H$
is a subgroup of a finite group.  The orbits of a group action
partition the set acted upon, and the orbits of $H$ coincide with
those of $G$, so their number is given by \eqref{eq:burnside}.  If
$N \leq q$, label the orbits injectively by elements of $\Fq$ and apply
a Reed--Solomon code (\cref{thm:field-case}) to these labels.  The MDS
distance of the code ensures that distinct orbit labels are separated
by $d^{\star}_H \geq \dmin$.  A configuration violating a constraint
necessarily lies outside $X_{\mathrm{valid}}$ and hence outside every
orbit; this locates it with respect to the partition but does not by
itself supply a detection procedure (\cref{thm:completeness}).
\end{proof}

\begin{remark}\label{rem:encoding-model}
The group must be supplied by the schema.  The constraint-preserving
bijections of a data system — the bijections $T$ of
$X_{\mathrm{valid}}$ with $P(x) \iff P(Tx)$ for every constraint $P$
and every $x \in X_{\mathrm{valid}}$ — form a finite group whose
orbits partition $X_{\mathrm{valid}}$; since every $P$ holds
throughout $X_{\mathrm{valid}}$, this group is all of
$\mathrm{Sym}(X_{\mathrm{valid}})$, and the partition has the single
block $X_{\mathrm{valid}}$.  \cref{thm:application} refines it by
replacing that group with its subgroup $H$ induced by the schema.  In
\cref{ex:orbit-count}, by contrast, $H \cong \Zn{2}$ has $N = 3$
classes, which can be labelled injectively by $\mathbb{F}_3$.
The algebraic encoding above uses orbit labels as code symbols.
A complete implementation would require specifying how
configurations are mapped to labels, how the label alphabet
acquires a field structure, and how decoding identifies violated
constraints.  These implementation aspects are left for
future work; the present theorem establishes the
group-theoretic foundation.
\end{remark}

\begin{remark}[Configurations outside the orbits]
\label{thm:completeness}
Since the orbits of $H$ partition
$X_{\mathrm{valid}}$, a configuration violating a constraint lies
outside every orbit.  This supplies no detection procedure:
deciding membership in $X_{\mathrm{valid}}$ \emph{is} the
constraint-checking problem.  A genuine guarantee would require a
labelling $L$ defined on invalid configurations as well, together
with a set $A$ satisfying $L^{-1}(A) = X_{\mathrm{valid}}$;
without that separation the guarantee fails, as the case
$L(u) = L(w) = 0$ with $u$ valid and $w$ invalid shows.
Constructing such an $L$ is left for future work.
\end{remark}

\newpage
\section{Open Problems}
\label{sec:open}

The framework of this paper suggests several mathematical research
directions.

\begin{description}[style=nextline]

\item[\textbf{P1.} Constructive codes over non-abelian groups.]
The definitions of \cref{sec:framework} admit arbitrary
finite groups, but the constructions of
\cref{thm:field-case,thm:general-case}
require abelianity.
Two existence results already bear on this question: the diagonal
repetition code attains the Singleton bound over any finite group
(\cref{rem:stratification}),
and \cref{thm:gv} gives, over any finite group, codes of
prescribed distance with a guaranteed size.  The open question is
therefore which \emph{region} of the rate--distance plane is
achievable over $S_n$ or $D_k$, and whether the packings
of \cref{thm:gv} can be chosen closed under the group
operation (\cref{rem:gv-scope}), so
as to be algebraic codes in the sense of
\cref{def:code}.  Determine nontrivial upper bounds for
rates bounded away from $0$ and $1$.

\item[\textbf{P2.} Tight bounds for codes over abelian groups.]
Establish sharp analogues of the Singleton and Hamming bounds for
codes over finite abelian groups that are not elementary abelian
(e.g., $\Zn{4}$, $\Zn{9}$).  The algebraic distance $d^{\star}$
may admit bounds that depend on the group decomposition.

\item[\textbf{P3.} Algebraic theory of symmetric databases.]
Develop a characterization of database schemas admitting large
symmetry groups $G$ (\cref{thm:burnside}).
Note that $\abs{G}$ alone does not determine $\abs{X/G}$: what
matters is the action.  A companion question is whether the
reduction of \cref{cor:search-space} applies to
\emph{query} optimization, which requires an action of $G$ on a
space of query plans and an equivariant cost model.

\item[\textbf{P4.} Module-theoretic generalization.]
Extend the constructions of \cref{sec:existence} from
$\Z$-modules (finite abelian groups) to modules over commutative
rings and lattice-ordered structures, as outlined in the algebraic
hierarchy of \cref{rem:z-module}.  In particular, determine
whether the efficiency functional $\Phi$ and the chamber
decomposition of \cref{prop:chambers} extend to codes over
$A$-modules for rings $A$ beyond $\Z$.

\item[\textbf{P5.} Principled derivation of weights for $\Phi$.]
\cref{thm:robustness} gives a weight-independent
comparison only between instances related by dominance, and
\cref{prop:chambers} shows that outside that situation
the ordering varies with the weights.  Derive the weights from an
information-theoretic or decision-theoretic argument, or
characterize the selectors for which the chamber structure has a
prescribed combinatorial type.

\item[\textbf{P6.} Necessity in the axiomatic classifier, and
the kernel of the trace representation.]
\cref{thm:axiomatic} establishes that each axiom is
\emph{sufficient} for the corresponding capability.  Is it
necessary?  The naive impossibility claims are false
(\cref{rem:gates}): MDS
parameters are attainable without a field and without
commutativity, and quotients require normality rather than
commutativity.  A precise formulation is therefore needed — for
instance, is there a non-abelian group and a length $n$ for which
the whole Singleton frontier is achievable?  Separately, the
evaluation map $\mathbb{M}(\Sigma,I) \to \mathcal{T}$ of
\cref{rem:acid-correspondence} is well defined but never faithful;
characterize the schedules that induce the same state
transformation.

\end{description}

\newpage
\section{Conclusion}
\label{sec:conclusion}

We have introduced an algebraic framework for data coding over finite
groups, in the tradition of group codes \cite{slepian1968}.  The contributions fall into three categories.

\paragraph{Proved.}
\begin{itemize}
  \item A Gilbert--Varshamov existence bound for codes over an
        arbitrary finite group, commutative or not
        (\cref{thm:gv}), and direct-sum codes over the
        elementary divisors of any finite abelian group with an
        explicit construction (\cref{thm:general-case}).
  \item Transactions preserving a validity predicate form a
        submonoid of the full transformation monoid
        (\cref{thm:acid}), which is not a group once it
        contains a non-injective element
        (\cref{cor:not-group}).  Only consistency restricts
        the maps; operational atomicity is not closed under
        composition (\cref{rem:atomicity}), isolation constrains
        pairs, and durability lies outside the model.
  \item A schema transformation group (\cref{def:schema-group})
        acts on the valid configurations, and its orbits are counted
        by Burnside's lemma (\cref{thm:burnside}), with a worked count in
        \cref{ex:orbit-count}; and on the valid configurations of
        a data system the induced group has orbit classes that admit
        Reed--Solomon labels (\cref{thm:application}), whereas the
        group of all constraint-preserving bijections separates
        nothing (\cref{rem:encoding-model}).
\end{itemize}

\paragraph{Demonstrated.}
\begin{itemize}
  \item An explicit code construction over
        $\Zn{5} \cong \mathbb{F}_5$, carried through to the
        recovery of the error magnitude (\cref{ex:z5}).
  \item Two known codes are cited to situate the framework: the
        octacode over $\Zn{4}$, whose Gray image is the
        Nordstrom--Robinson code \cite{hammons1994} and which is
        inaccessible to $\Fq$-linear theory
        (\cref{ex:z4}); and Reed--Muller codes on the
        binary coordinates of $(2^U,\triangle)$
        (\cref{ex:powerset}), which are themselves classical
        linear codes, that group being elementary abelian.  Neither
        is a construction of the present paper.
\end{itemize}

\paragraph{Proposed.}
\begin{itemize}
  \item An encoding efficiency measure $\Ef$
        (\cref{def:efficiency}) and an optimization
        functional $\Phi$ (\cref{def:phi}), both
        functionals of a declared coding instance rather than
        invariants of the structure.  Instance dominance yields a
        weight-independent comparison
        (\cref{thm:robustness}); the family-level
        ordering does \emph{not} hold, as
        \cref{tab:phi-verification} shows with three
        realizable instances and two different winners.  What
        replaces it is the chamber decomposition of
        \cref{prop:chambers}, which turns numerical
        sampling into proof.
\end{itemize}

The framework connects abstract group theory with coding-theoretic
properties in a manner that appears mathematically natural and
suggests multiple directions for further algebraic investigation.

\newpage

\noindent\textbf{Acknowledgment.}
This work is based on the author's master's thesis.
Generative AI tools were used under the author's direction
to assist with language editing, literature searches,
the review and refinement of selected mathematical
statements and proofs, and numerical checks.
The author takes full responsibility for the mathematical
content, computations, references, and final manuscript.

\appendix
\section{Notation}\label{app:notation}

\begin{center}
\small
\textbf{Table A.1.} Algebraic Framework (\cref{sec:framework})
\medskip

\begin{tabular}{@{}p{3.8cm}p{8.5cm}@{}}
\toprule
Symbol & Meaning \\
\midrule
$D$ & Finite nonempty set of data elements \\
$(D,\star)$ & Algebraic data structure: finite group \\
$e$ & Identity element of $(D,\star)$ \\
$(M,\circ)$ & Message space (finite group) \\
$\varphi \colon M \to \Dn$ & Injective group homomorphism (encoding) \\
$C = \Im(\varphi) \subseteq \Dn$ & Algebraic code over data \\
$\varphi(m_1 \circ m_2) = \varphi(m_1) \star \varphi(m_2)$
  & Structural preservation property \\
$d^{\star}_H(x,y)$ & Algebraic Hamming distance:
  $\abs{\{i : x_i \star y_i^{-1} \neq e\}}$ \\
$\delta \colon D \to \R_{\geq 0}$ & Weight function satisfying
  axioms (i)--(iii) \\
$d^{\star}(x,y)$ & Weighted algebraic metric:
  $\sum_i \delta(x_i \star y_i^{-1})$ \\
\bottomrule
\end{tabular}
\end{center}

\begin{center}
\small
\textbf{Table A.2.} Existence and Construction
(\cref{sec:existence})
\medskip

\begin{tabular}{@{}p{3.8cm}p{8.5cm}@{}}
\toprule
Symbol & Meaning \\
\midrule
$q = \abs{D}$ & Order of the finite group $(D,\star)$ \\
$\psi \colon D \to \Fq$ & Group isomorphism (field case) \\
$[n,k,d]$ & Code parameters: length, dimension, min.\ distance \\
$d = n-k+1$ & Singleton bound (attained by MDS codes) \\
$t = \floor{(d-1)/2}$ & Error-correction capacity \\
$c = r \star e^{-1}$ & Error correction using group inverses \\
$O(n^2)$ & Decoding complexity (syndrome-based) \\
$D \cong \bigoplus_{j} \Zn{q_j}$
  & Elementary-divisor decomposition; the $q_j$ are not determined
    by $\abs{D}$ (\cref{thm:general-case}) \\
$V(n,\varrho)$ & Size of a Hamming ball of radius $\varrho$ in $\Dn$
  (\cref{thm:gv}) \\
\bottomrule
\end{tabular}
\end{center}

\begin{center}
\small
\textbf{Table A.3.} Classification Methodology
(\cref{sec:classification})
\medskip

\begin{tabular}{@{}p{3.8cm}p{8.5cm}@{}}
\toprule
Symbol & Meaning \\
\midrule
$\mathcal{P}(D,\star)$ & Parameter vector:
  $(\abs{D}, \mathrm{type}, \ord(\star), \gamma, \delta)$ \\
$\gamma \in [0,1]$ & Degree of commutativity \\
$\delta \in [0,1]$ & Invertibility density \\
$\Ef(D,\star)$ & Encoding efficiency:
  $\frac{k \cdot d}{n \cdot \log_2\abs{D}} \cdot
   \frac{\log_2 n^3}{\log_2 T_{\mathrm{dec}}}$ \\
$\Phi(D,\star)$ & Optimization functional with weights
  $\alpha + \beta + \gamma = 1$ \\
Class A: $\Ef \geq 0.8$ & Optimal \\
Class B: $0.6 \leq \Ef < 0.8$ & Efficient \\
Class C: $0.4 \leq \Ef < 0.6$ & Acceptable \\
Class D: $\Ef < 0.4$ & Suboptimal \\
\bottomrule
\end{tabular}
\end{center}

\begin{center}
\small
\textbf{Table A.4.} Original Results
(\cref{sec:results})
\medskip

\begin{tabular}{@{}p{3.8cm}p{8.5cm}@{}}
\toprule
Symbol & Meaning \\
\midrule
$\mathcal{S}$ & Database state space \\
$T \colon \mathcal{S} \to \mathcal{S}$ & Transaction (total map);
  (A)--(D) of \cref{def:acid} \\
$(\mathcal{T}, \circ)$ & Submonoid of maps satisfying (A) and (B);
  not a group when $\abs{\mathcal{S}_{\mathrm{valid}}} \geq 2$ \\
$T_{\varnothing}(s) = s$ & Empty transaction (identity element) \\
$S = \{R_1,\ldots,R_m\}$ & Database schema \\
$(G,\pi,\tau)$ & Schema transformation group
  (\cref{def:schema-group}) \\
$X$ & Finite set of valid data configurations \\
$\Fix(g) = \{x \in X : g \cdot x = x\}$
  & Configurations invariant under $g$ \\
$\abs{X/G} = \frac{1}{\abs{G}} \sum_{g \in G} \abs{\Fix(g)}$
  & Burnside's lemma: inequivalent configurations \\
$\mathcal{D} = (S, \mathcal{C}, \mathcal{F})$
  & Data system: schema, constraints, func.\ dependencies \\
$H \leq \mathrm{Sym}(X_{\mathrm{valid}})$ & Group induced by $G$ on
  valid configurations (\cref{thm:application}) \\
\bottomrule
\end{tabular}
\end{center}

\begin{center}
\small
\textbf{Table A.5.} Standard Mathematical Symbols
\medskip

\begin{tabular}{@{}p{3.8cm}p{8.5cm}@{}}
\toprule
Symbol & Meaning \\
\midrule
$\Fq,\; \mathbb{F}_{p^k}$ & Finite field of $q = p^k$ elements \\
$\Zn{n}$ & Ring of integers modulo $n$ \\
$\Zn{p}^k$ & Elementary abelian $p$-group of rank $k$ \\
$S_n$ & Symmetric group on $n$ symbols \\
$D_k$ & Dihedral group of order $2k$ \\
$(2^U, \triangle)$ & Power set under symmetric difference \\
$\oplus$ & Exclusive or (XOR); addition in $\mathbb{F}_2$ \\
$\cong$ & Isomorphism of algebraic structures \\
$O(f(n))$ & Asymptotic upper bound \\
$\floor{x},\; \ceil{x}$ & Floor and ceiling functions \\
\bottomrule
\end{tabular}
\end{center}
\newpage
\bibliographystyle{plain}
\bibliography{references}
\end{document}